\documentclass[11pt]{amsart}
\usepackage{amsmath,amsthm,amssymb,amsfonts,enumerate,color,esint, graphicx, pgf,tikz}
\usetikzlibrary{intersections}
\pgfdeclarelayer{background}
\pgfsetlayers{background,main}
\usepackage[latin1]{inputenc}
\usepackage{bbm}

\newcommand{\R}{\mathbb{R}}

\newcommand{\h}{\mathcal{H}}

\newcommand{\uu}{\hat{u}}
\newcommand{\ur}{\hat \gamma}
\newcommand{\bu}{\bar u}

\newcommand{\vep}{\varepsilon}
\newcommand{\supp}{\operatorname{supp}}

\newcommand{\loc}{\operatorname{loc}}

\definecolor{ffqqqq}{rgb}{1,0,0}
\definecolor{qqqqff}{rgb}{0,0,1}

\newtheorem{thm}{Theorem}[section]

\newtheorem{cor}[thm]{Corollary}
\newtheorem{lem}[thm]{Lemma}
\theoremstyle{definition}
\newtheorem{defn}[thm]{Definition}
\newtheorem{rem}[thm]{Remark}

\numberwithin{equation}{section}

\allowdisplaybreaks

\author[D.~Kriventsov]{Dennis Kriventsov}
    \address{Dennis Kriventsov. Department of Mathematics\\
    Rutgers University\\
110 Frelinghuysen Rd., Piscataway, NJ 08854, USA}
    \email{dennis.kriventsov@rutgers.edu}
 
\author[M.~Soria-Carro]{Mar\'ia Soria-Carro}
    \address{Mar\'ia Soria-Carro. Department of Mathematics\\
    Rutgers University\\
110 Frelinghuysen Rd., Piscataway, NJ 08854, USA}
    \email{maria.soriacarro@rutgers.edu}

\title[An Elliptic-Parabolic Free Boundary Problem]{An Elliptic-Parabolic Free Boundary Problem with Discontinuous Data}



\begin{document}

\begin{abstract}
We consider an elliptic-parabolic free boundary problem that models the fluid flow through a partially saturated porous medium. The free boundary arises as the interface separating the saturated and unsaturated regions. Our main goal is to investigate, for the 1+1 dimensional model, how jump discontinuities on the boundary and initial data influence the regularity of both the solution and the free boundary. We show that if the data is merely bounded, then weak solutions are Lipschitz in space and $C^{1/2}$ in time in the unsaturated region. Moreover, the free boundary is locally the graph of a $C^{1/2}$ function, and this regularity is optimal. We view this analysis as a stepping stone towards the study of local regularization for higher-dimensional elliptic-parabolic free boundaries.
\end{abstract}

\maketitle

\section{Introduction}

The slow flow of an incompressible fluid through a one-dimensional homogeneous partially saturated porous medium is described by the quasilinear equation
\begin{equation} \label{eq:main}
\partial_t c(u) - u_{xx}=0,
\end{equation}
where $c: \R\to\R$ is a Lipschitz continuous function satisfying 	
\begin{equation} \label{eq:c}
c(s)=0 \ \text{ if } s\leq 0 \quad \text{and} \quad c'(s)>0 \ \text{ if } s>0.
\end{equation}
The quantity $u$ stands for the hydrostatic potential due to capillary suction and $c(u)$ for the relative moisture content. 
Here, the porous medium is saturated whenever $c$ is identically zero, and unsaturated otherwise. Formally, writing the equation as
\begin{equation} \label{eq:main2}
c'(u) u_t - u_{xx}=0,
\end{equation}
we observe that \eqref{eq:main2} is elliptic whenever $u<0$, and parabolic whenever $u>0$. This leads to an elliptic-parabolic free boundary problem, where the zero-level set of $u$ splits the domain into the saturated and unsaturated regions. For a precise derivation of the model see~\cite{B, vDP}.

\subsection{Brief history}  The study of such problems started in the early 1980s with the work of Fasano and Primicerio~\cite{FP}, and van Duyn and Peletier~\cite{vDP,vD}, who considered constant Dirichlet boundary conditions. The latter introduced a notion of weak solution and developed a comprehensive theory including existence, uniqueness, a maximum principle, regularity, and continuity of the free boundary. More details are provided in Section~\ref{sec:classical}.

Hulshof and Peletier~\cite{HP,H} extended these results to Robin boundary conditions with Lipschitz continuous data, and Bertsch and Hulshof~\cite{BH} proved higher regularity of weak solutions under additional assumptions on $c(u)$ and the initial data; see also~\cite{vDH}. 

Around the same time, Alt and Luckhaus~\cite{AL} established existence and uniqueness of weak solutions to more general elliptic-parabolic systems in any dimension.  
For higher-dimensional flows satisfying \eqref{eq:main}-\eqref{eq:c}, DiBenedetto and Gariepy~\cite{DiBG} proved local continuity of $c(u)$, and Hulshof and Wolanski~\cite{HW} studied the regularity of the free boundary for monotone flows.

In the 2000s, Mannucci and V\'azquez~\cite{MV} introduced a viscosity solution framework for the one-dimensional setting, which was later extended to fully nonlinear equations in higher dimensions by Kim and Po\v{z}\'ar~\cite{KP}.

For further developments, see \cite{D1,D2}, and the references therein.

\subsection{Challenges} The main difficulty in this model arises from the degeneracy of the equation in the saturated region, where time plays the role of a parameter. As a result, the time regularity in this region cannot be better than the regularity of the boundary conditions.
In fact, even with continuous data, it is possible to construct solutions that are discontinuous in time. For example, taking constant negative Dirichlet boundary data and positive initial data, the solution becomes stationary after a finite time. 

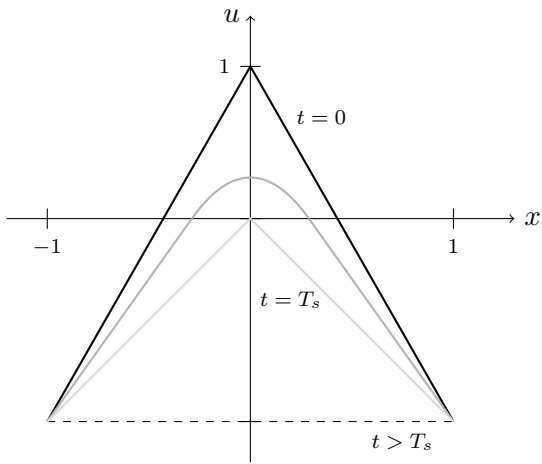
\begin{figure}[h] 
\centering
\begin{tikzpicture}[scale=2.7]

\draw[->] (-1.2, 0) -- (1.3, 0) node[right] {$x$};
\draw[->] (0, -1.2) -- (0, 1) node[left] {$u$};

\foreach \x in {-1,1}
  \draw (\x, 0.05) -- (\x, -0.05) node[below] {\scriptsize $\x$};
\draw (0.05, -1) -- (-0.05, -1); 
\draw (0.05, 0.75) -- (-0.05, 0.75) node[left] {\scriptsize $1$};
\draw (0.05, 0) -- (-0.05, 0); 

\draw[thick] (-1,-1) -- (0,0.75) -- (1,-1);
\node at (0.35, 0.5) {\scriptsize $t = 0$};

\draw[dashed] (-1,-1) -- (1,-1);
\node at (0.75, -1.1) {\scriptsize $t > T_s$};

\def\a{0.6}   
\def\h{0.3}   

\pgfmathsetmacro{\xzero}{\a*sqrt(\h/(\h+1))}

\pgfmathsetmacro{\m}{(-1-0)/(-1-(-\xzero))}

\pgfmathsetmacro{\hnew}{\m*\xzero/2}

\draw[thick, gray!60, domain=-\xzero:\xzero, samples=100] plot(\x, {\hnew*(1 - (\x/\xzero)^2)});

\draw[thick, gray!60] (-\xzero,0) -- (-1,-1);
\draw[thick, gray!60] (\xzero,0) -- (1,-1);

\draw[thick, gray!30] (-1,-1) -- (0,0) -- (1,-1);
\node at (0.2, -0.4) {\scriptsize $t = T_s$};

\end{tikzpicture}
\caption{Discontinuous solution at saturation time $t=T_s$.}
\label{fig1}
\end{figure}
As shown in Figure~\ref{fig1}, the parabolic region collapses at  $t=T_s$, and  after this time, the solution jumps to the stationary boundary data~\cite{AL}.
However, if $u$ is bounded, then $c(u)$ must be continuous, which is a local property satisfied by this function \cite{DiBG}.

On the other hand, heuristically, $u$ satisfies the two-phase free boundary problem
\begin{equation} \label{pb:fbp}
    \begin{cases}
c'(u) u_t -u_{xx}=0 & \hbox{ in } \{u>0\},\\
u_{xx}=0 & \hbox{ in } \{u<0\},\\
u_x^+ = u_x^- & \hbox{ on } \{u=0\},
\end{cases}
\end{equation}
where $u_x^+$ and $u_x^-$ are the one-sided spatial derivatives of $u$  as $x$ approaches the free boundary from the positive and negative phase, respectively.
Since $u$ is harmonic in $\{u<0\}$, then in this set, $u^-$ is a linear function with respect to $x$ for almost every time. From the free boundary condition, $u_x^+=u_x^-$, it follows that the Dirichlet data for $u^-$ becomes the Neumann data for $u^+$. Consequently, jump discontinuities on the boundary data will propagate into the interior of the domain. 
These nonlocal effects significantly increase the difficulty of the analysis.  

If $c$ is convex near $0$ and $c'(0)=0$, then $u$ is a classical solution to \eqref{eq:main} (see \cite{BH}).
In this work, we will focus on the case $c'\geq c_0>0$ on $\{u>0\}$. 
 Although this condition introduces a jump discontinuity in $c'$ at $0$, it ensures that the equation remains uniformly parabolic in the unsaturated region.
 In this case,  \eqref{pb:fbp} can also be understood as the limiting problem for the evolution of two phases with different time scales of diffusion, considered in \cite{E,KSC}.

It is also worth noting that, in the positive phase,
$u$ satisfies a parabolic Bernoulli problem with (possibly discontinuous) time-dependent free boundary condition. Problems of this type arise in combustion theory and were first introduced in \cite{CV}.

Our overarching goal is to understand the extent to which elliptic-parabolic free boundaries regularize locally under the evolution. Unlike for other geometric evolution equations, such as mean curvature flow, Stefan problems, Bernoulli-type problems, and porous medium equations, here there are no genuinely local results in the literature concerning the smoothness of the free boundary. What is available is either global in nature (see \cite{HW}) or only gives information about the solution in the parabolic phase, not the interface (see \cite{DiBG}). 

In light of the nonlocal nature of the elliptic phase, this is not surprising: perhaps a collapse like in Figure \ref{fig1} occurs somewhere else and then immediately impedes the smoothing of the interface via a discontinuity in the elliptic phase. In this paper, we ask: if a jump in the elliptic phase does occur for global reasons, what are the local effects of this on the interface? In one dimension, we are able to answer this completely. We leave open the interesting question of how this extends to higher dimensions, where it makes sense to look not just at worst-case but at \emph{generic} behavior of the free boundary.

\subsection{Main assumptions and results} 
Assuming only that the data is bounded by some universal constants, we aim to find local properties that depend only on these universal bounds. 
More precisely, given $0< \lambda \leq\Lambda$, we assume the following conditions on the data:
\begin{enumerate}
\item[$(H_1)$] $c(u)=u^+:=\max\{u,0\}$; 
\item[$(H_2)$] $-\Lambda \leq u(-1,t)\leq -\lambda$ and $\lambda \leq u(1,t)\leq \Lambda$  for a.e.  $t\in  [-1,0]$;
\item[$(H_3)$] $u^+(x,-1)=0$ for a.e.~$x\in [-1,0]$ and $0<u^+(x,-1) \leq \Lambda$ for a.e.~$x\in (0,1]$.
\end{enumerate}
In $(H_1)$, we consider the model case 
of a function satisfying the condition \( c' \geq c_0 > 0 \) on \( \{u > 0\} \). This choice simplifies the analysis while still capturing the essential features of the problem. In $(H_2)$, the sign condition on the boundary data prevents the medium from becoming fully saturated in finite time. In $(H_3)$, the sign condition on the initial data guarantees that the free boundary will be a graph for future times. 

For $r>0$, let $Q_r:= (-r,r) \times (-r^2,0]$. 
Given a parabolic domain $Q\subseteq \R^{n+1}$, we define $C_{x,t}^{1,1/2}(\overline{Q})$ as the space of continuous functions $u$ on $\overline{Q}$ such that
$$
  [u]_{C_{x,t}^{1,1/2}(\overline Q)} :=\sup_{\substack{(x,t), (y, s) \in \overline Q\\ (x,t) \neq (y,s)}} \frac{|u(x,t)-u(y,s)|}{( |x-y|^2 + |t-s| )^{1/2}}< \infty.
$$
We call a constant $C>0$ universal if it depends only on $\lambda$ and $\Lambda$. 

Our main result is the following:  

\begin{thm}\label{thm:main}
Let $u$ be a weak solution to $\partial_t c(u) - u_{xx}=0$ in $Q_1$ satisfying the assumptions $(H_1)-(H_3)$. There are universal constants $K,L,M,N>0$ such that the following hold:
\begin{enumerate}[$(i)$]
\item (Regularity of $u$). $u$ satisfies the estimates
$$
\|u_x\|_{L^\infty(Q_{1/2})} \leq L \quad \text{and}\quad [u^+]_{C_{x,t}^{1,1/2}(\overline{Q_{1/2}})} \leq M.
$$

\item (Non-degeneracy). There is a universal constant $0<\rho_0<1/2$ such that
$$
u_x \geq N \quad \text{ a.e. in } Q_{\rho_0}(x_0,t_0),
$$
for any $(x_0,t_0) \in \{u=0\}\cap Q_{1/2}$.\medskip

 \item (Regularity of the free boundary). 
There is a universal constant $0<\delta_0 <1/2$, and a continuous function 
$\gamma: [-1/4,0]\to [-1+\delta_0,1-\delta_0]$, 
such that 
$$\{u=0\}\cap Q_{1/2} = \{ (x,t)\in Q_{1/2} : x = \gamma(t)\}.$$
Furthermore, $\gamma\in C^{1/2}([-1/4, 0])$ with 
$$[\gamma]_{C^{1/2}([-1/4,0])}\leq K,$$
 and this is the optimal regularity.
\end{enumerate}
\end{thm}

We refer the reader to Section~\ref{sec:classical} for the definition of weak solution. 
As discussed above, the lack of regularity in the data presents several challenges. For example, the classical methods for constructing weak solutions do not work in this setting (see Remark~\ref{rem:classical}), and establishing local regularity estimates is particularly difficult due to the nonlocal nature of the problem.

\subsection{Ideas and methodology} 
Our general strategy to address these difficulties consists of regularizing the initial and boundary data, and constructing approximating weak solutions via the methods outlined in Section~\ref{sec:mainstrategy}. To pass to the limit, we derive an energy estimate (Lemma~\ref{lem:energy}) that depends only on the universal bounds $\lambda$ and $\Lambda$ (see Section~\ref{sec:approxweak}).
This method suggests that to prove Theorem~\ref{thm:main}, it is enough to obtain quantitative universal estimates for the approximating weak solutions that will be preserved at the limit. We justify this fact rigorously in Lemma~\ref{lem:limit}. 

First, we establish the Lipschitz bound for the spatial derivative (Theorem~\ref{thm:lipschitz}). Working in one spatial dimension provides a key simplification: since $u$ is linear in the negative phase, for almost every time, then $u_x^-$ is bounded a.e., which in turn bounds $u_x^+$ on $\{u=0\}$ (see Remark~\ref{rem:higherdim}). To propagate this estimate across the zero-level set, our main idea is to show that the free boundary remains uniformly separated from the lateral boundary (Lemma~\ref{lem:unifsep}). This separation allows us to apply interior parabolic regularity and a local in time maximum principle to obtain the Lipschitz bound up to the free boundary. Moreover, on the parabolic side, we refine the time regularity using the global equation satisfied by $u^+$ (Theorem~\ref{thm:timereg}).

The non-degeneracy (Theorem~\ref{thm:monotonicity}) is one of the most delicate steps in the proof. A priori, the free boundary is only known to be a continuous graph, so classical tools like the parabolic Hopf Lemma are not applicable.
Our approach combines an analysis of the behavior of level sets of caloric functions with a careful construction of suitable barriers.  Thanks to this uniform spatial monotonicity, we can apply the hodograph transform to $u$, and derive boundary estimates for the transformed problem, implying the $C^{1/2}$ time regularity of the free boundary (Theorem~\ref{thm:interface}). To the best of our knowledge, this method is new in the context of elliptic-parabolic free boundary problems.
Finally, to justify the optimality, we produce a solution, using a subtle fixed point argument, for which the free boundary satisfies $\gamma(t)\sim \sqrt{t}$ (Theorem~\ref{thm:optimal}).

\medskip
The paper is organized as follows. In Section~2, we give a brief summary of the main results of the classical theory of weak solutions with Lipschitz data. In Section~3, we derive the energy estimate and introduce the approximating weak solutions. We prove the regularity estimates from part~$(i)$ in Section~4, the non-degeneracy from part~$(ii)$ in Section~5, and the optimal regularity of the free boundary from part~$(iii)$ in Section~6.

\section{Weak theory for Lipschitz data} \label{sec:classical}

In this preliminary section, we summarize the main classical results for weak solutions with Lipschitz continuous data. 
For $T>0$, consider the Cauchy-Dirichlet problem:
\begin{equation} \label{pb:main}
\begin{cases}
\partial_t c(u) - u_{xx} =0 & \text{ in } Q_T:=(-1,1)\times (0,T],\\
u(-1,t) = f(t), \ u(1,t)= g(t) & \text{ for } 0<t \leq T,\\
c(u(x,0)) = v_0(x) & \text{ for } -1 \leq x \leq 1. 
\end{cases}
\end{equation}
The functions $c$, $f$, $g$, and $v_0$ satisfy the following assumptions:
\begin{enumerate}
\item[$(A_1)$] $c: \R\to\R$ is a Lipschitz function on $\R$ and differentiable on $(0,\infty)$ satisfying
$c(s)=0$ if $s\leq 0$ \text{and}  $c'(s)>0$  if $s>0.$
\item[$(A_2)$] $f, g:[0,T] \to \R$ are Lipschitz functions such that $f<0$ and $g>0$ on $[0,T]$.
\item[$(A_3)$] There exists a Lipschitz function $u_0 : [-1,1] \to \R$ such that $c(u_0)=v_0$ on $[-1,1]$. Moreover,
$u_0(-1)= f(0)$ and $u_0(1)=g(0)$. 
\item[$(A_4)$] In addition, $u_0 \in C^{2+\beta}([-1,1])$ for some $0<\beta\leq 1$, $u_0''(-1)=u_0''(1)=0$, and $u_0'' \geq - \kappa c'(u_0)$ whenever $u_0 \neq 0$, for some $\kappa>0$.
\item[$(A_5)$] There is  $\gamma_0 \in (-1,1)$ such that $v_0(x)=0$ if $-1\leq x\leq \gamma_0$ and $v_0(x)>0$ if $\gamma_0< x \leq 1$.  
\end{enumerate}

We recall that $u: Q_T \to \R$ belongs to the energy space $L^2(0,T; H^1(-1,1))$ if 
$$
\iint_{Q_T} u^2 \, dxdt + \iint_{Q_T} u_x^2 \, dxdt<\infty. 
$$

\begin{defn}[Weak solution] 
A function $u \in L^2(0,T; H^1(-1,1))$ is called a weak solution to \eqref{pb:main} if the following holds:
\begin{enumerate}[$(i)$]
\item $c(u)\in C(0,T; L^2(-1,1))$;
\item $u(-1,t)=f(t)$ and $u(1,t)=g(t)$ for a.e.~$t\in [0,T]$;
\item $u$ satisfies the identity:
$$
\iint_{Q_T} ( u_x \varphi_x - c(u) \varphi_t) \, dxdt= \int_{-1}^1 v_0(x) \varphi(x,0) \, dx,
$$
for any $\varphi \in  C^1(\overline{Q_T})$ vanishing at $x=-1$, $x=1$ and $t=T$.
\end{enumerate}
\end{defn}

\subsection{Classical theory} 
The theory of weak solutions to \eqref{pb:main} was developed by van Duyn and Peletier \cite{vDP,vD} for $f\equiv -1$ and $g\equiv 1$, and extended by Hulshof and Peletier \cite{HP,H} for Robin boundary conditions with Lipschitz continuous data; see also \cite{BH}.

We summarize the main results in the next theorem.

\begin{thm}\label{thm:classical}
Assume the conditions $(A_1)-(A_3)$ hold. Then:
\begin{enumerate}[$(i)$] 
\item (Existence and Uniqueness). There exists a unique weak solution to \eqref{pb:main}.
\item (Improved regularity). The weak solution satisfies $$u\in L^\infty(0,T;W^{1,\infty}(-1,1))\cap L^2(0,T; H^2(-1,1))
\ \text{ and } \ c(u)\in C(\overline{Q_T}).$$
\item (Free boundary). Assume that $(A_4)$ and $(A_5)$ hold.
Then there exists a continuous function $\gamma : [0,T] \to (-1,1)$ such that $\gamma(0)=\gamma_0$, and for each $t\in [0,T]$, we have
\begin{align*}
c(u(x,t))=0 \text{ for } -1<x\leq \gamma(t) \ \text{ and } \ c(u(x,t))>0 \text{ for } \gamma(t) < x < 1.
\end{align*}
\end{enumerate}
\end{thm}

An important tool in the study of weak solutions is the comparison principle, which holds for the function $c(u)$; see \cite[Theorem~3]{vDP,HP}.

\begin{thm}[Comparison principle] \label{thm:comparison}
Let $c$ satisfy $(A_1)$.
Let $u_1$ and $u_2$ be the weak solutions to \eqref{pb:main} with data $v_{01}$, $f_1$, $g_1$, and $v_{02}$, $f_2$, $g_2$, respectively, which satisfy $(A_2)-(A_3)$. 
If $v_{01} \leq v_{02}$ on $[-1,1]$, $f_1 \leq f_2$ and $g_1\leq g_2$ on $[0,T]$, then $c(u_1)\leq c(u_2)$ in ${Q_T}$.
\end{thm}

If, in addition, $\partial_t c(u_1), \partial_t c(u_2) \in L^2(Q_T)$, then the comparison principle holds for $u_1$ and $u_2$. This result was proved by Alt and Luckhaus \cite[Theorem 2.2]{AL}. 

\subsection{An elliptic-parabolic free boundary problem} \label{sec:ellipticparabolic}
In the particular case of
 $$c(u)=u^+,$$ we deduce from Theorem~\ref{thm:classical} that
$u$ satisfies the elliptic-parabolic free boundary problem
\begin{equation} \label{eq:FBP}
\begin{cases}
u_{xx} = 0 & \text{ for } -1<x<\gamma(t), \ 0<t\leq T, \\
 u_t - u_{xx} = 0 & \text{ for } \gamma(t) < x < 1, \ 0<t\leq T,\\
u(\gamma(t),t)=0 & \text{ for } 0<t\leq T,\\
u_x^+ (\gamma(t), t) = u_x^- (\gamma(t),t) & \text{ for } 0<t\leq T,
\end{cases}
\end{equation}
in the strong sense, 
where
 $$
 u_x^\pm(\gamma(t),t) = \lim_{x\to \gamma(t)^\pm} u_x(x,t) \quad  \text{ exists a.e. }
 $$
 Since $c$ satisfies the assumptions of \cite[Theorems~5 and 9]{vDP}, we further have
$$
u(x,t) = -\frac{f(t)}{1+ \gamma(t)}(x+1)+ f(t) \quad \text{ for all } -1\leq x < \gamma(t), \ 0\leq t\leq T,
$$
and $u$ satisfies the heat equation in the classical sense for all $\gamma(t) < x < 1$ and $0<t\leq T$.

This particular case will be relevant in Section~\ref{sec:approxweak}.

\subsection{Main strategy} \label{sec:mainstrategy}
The authors in \cite{vDP,HP} prove existence of weak solutions by regularizing the data, solving a sequence of approximating problems, and obtaining uniform estimates to pass to the limit. 
More precisely, for each $n\geq 1$, there is a unique smooth solution to 
\begin{equation}\label{pb:approx}
\begin{cases}
 c_n'(u_n) \partial_t u_n - \partial_{xx} u_n = 0 & \text{ in } Q_T,\\
u_n(-1,t)= f_n(t), \ u_n(1,t)=g_n(t) & \text{ for } 0<t\leq T,\\
u_n(x,0)=u_{0n}(x) & \text{ for } -1\leq x \leq 1,
\end{cases}
\end{equation}
where $c_n$, $f_n$, $g_n$, and $u_{0n}$ are smooth functions converging uniformly to $c$, $f$, $g$, and $u_0$ in their respective domains, and $1/n \leq c_n'\leq K$ for some $K>1$.
Moreover, for all $n\geq 1$,
$$
 |u_n| \leq M \quad \text{ and } \quad |\partial_x u_n| \leq L \quad \text{ on } \overline{Q_T},
$$
where $M>0$ depends only on $\|f_n\|_\infty$, $\|g_n\|_\infty$, and $\|u_{0n}\|_\infty$, and $L>0$, depends only on $K$, $\|f_n'\|_\infty$, $\|g_n'\|_\infty$, and $\|u_{0n}'\|_\infty$. In particular, if $v_n=c_n(u_n)$, then
$$
|\partial_x v_n| = |c_n'(u_n)| |\partial_x u_n| \leq K L \quad \text{ on } \overline{Q_T}.
$$
From these uniform bounds and the uniqueness of weak solutions to \eqref{pb:main}, it follows that
$$
u_n \to u  \text{ weakly in }  L^2(0,T; H^1(-1,1)),
$$
$$
v_n \to c(u)  \text{ uniformly on } \overline{Q_T}.
$$
Passing to the limit in the weak formulation of \eqref{pb:approx}, we get $u$ is the weak solution to \eqref{pb:main}. 

\begin{rem} \label{rem:classical}
A key point in this approach is that obtaining uniform bounds on $\partial_x u_n$ requires $f$, $g$, and $u_0$ to be Lipschitz continuous. These bounds are crucial for taking the limit, and also for proving higher regularity of weak solutions and properties of the free boundary.
\end{rem}

\section{Some useful tools}

In this section, we introduce some useful notation, prove an energy estimate, and define the approximating weak solutions that will be fundamental in the following sections.

\subsection{Notation}

Given $r>0$ and $(x_0,t_0)\in \R^2$, we write 
$$Q_r(x_0,t_0)= (x_0-r,x_0+r)\times (t_0-r^2, t_0].$$
If $(x_0,t_0)=(0,0)$, we simply write $Q_r$. 

We define the saturated and unsaturated regions in $Q_r$ as 
$$Q_r^-(u):=Q_r \cap \{u<0\} \quad \text{ and } \quad Q_r^+(u):=Q_r \cap \{u>0\},$$ respectively.
We denote the zero-level set of $u$ in $Q_r$ as
$$\Gamma_r(u) :=Q_r\cap \{u=0\}.$$
When $r=1$, we will drop the subindex. 

Let $t_1<t_2$. Given $a,b \in C([t_1,t_2])$ satisfying $a(t)<b(t)$ for all $t\in [t_1,t_2]$, consider 
\begin{equation}\label{eq:domD}
D = \{ (x,t)\in \R^{2} : a(t)< x < b(t) \text{ and } t_1<t \leq t_2\}.
\end{equation}
We denote the parabolic boundary of $D$ as $\partial_p D$. The lateral boundary of $D$ is the set 
$$\partial_l D = \{ (a(t), t) : t\in (t_1,t_2]\}\cup\{ (b(t), t) : t\in (t_1,t_2]\},
$$
and the bottom of $D$ is the set  
$$\partial_b D = [ a(t_1), b(t_1) ] \times \{t_1\}.$$ 
We have that 
$
\partial_p D = \partial_l D \cup \partial_b D.
$

\subsection{Energy estimate} The following energy estimate will be useful.

\begin{lem}[Local energy estimate] \label{lem:energy}
Let $c$ be an increasing smooth function such that $c(0)=0$ and $c'\leq K$ on $\R$.
Let $u$ be a classical solution to $\partial_t c(u) - u_{xx}=0$ in $Q_1$. 
    Then 
    $$
    \sup_{-r^2<t<0} \int_{-r}^r c(u)^2  \, dx + \iint_{Q_{r}} u_x^2 \, dxdt
    \leq C \iint_{Q_1} u^2\, dxdt,
    $$
    for any $0<r\leq1/2$, where $C>0$ depends only on $K$.
\end{lem}

\begin{proof}
Multiplying the equation by $\varphi\in C_c^\infty(Q_1)$, and integrating in $Q_1$, we get
\begin{align*} 
    \iint_{Q_1} \big(\partial_t c(u) -  u_{xx} \big) \varphi \, dxdt =0.
\end{align*}
Fix $0<r\leq 1/2$. Let $\eta \in C^{\infty}_c (Q_1)$ be a space-time cut-off function such that $0\leq \eta\leq 1$, $\eta\equiv 1$ in $Q_{r}$, and $|\eta_x|+|\eta_t|\leq C_0$ in $Q_1$. 
Taking $\varphi = u \eta^2$, and integrating by parts in space,
\begin{align*}
  \iint_{Q_1} u_x^2 \eta^2 \, dxdt + 2\iint_{Q_1} u_x u  \eta \eta_x\, dxdt= -  \iint_{Q_1} \partial_t c(u) u \eta^2\, dxdt. 
\end{align*}
By Young's inequality, we have
\begin{align*}
    2\iint_{Q_1} u_x u  \eta \eta_x\, dxdt 
    &  \leq \frac{1}{2} \iint_{Q_1}   u_x^2 \eta^2\, dxdt +2{C_0^2} \iint_{Q_1} u^2 \, dxdt.
\end{align*}
Absorbing the first term on the left-hand side, we get
\begin{equation} \label{eq:bd1}
\frac{1}{2}  \iint_{Q_1}   u_x^2 \eta^2\, dxdt 
\leq  2C_0^2 \iint_{Q_1} u^2 \, dxdt - \iint_{Q_1} \partial_t c(u) u \eta^2 \, dxdt.
\end{equation}
It remains to bound the term with $\partial_t c(u)$. Define 
\begin{equation} \label{eq:G}
G(s) := \int_0^s c(r)\, dr.
\end{equation}
Note that $G'(s)=c(s)$. Moreover, since $|c(s)|\leq K |s|$ for all $s\in \R$, we have
$$
|G(s)|\leq \int_0^s |c(r)|\, dr \leq \frac{K}{2} s^2.
$$
Integrating by parts in time, 
\begin{align*}
- \iint_{Q_1} \partial_t c(u) u \eta^2 \, dxdt
&=  \iint_{Q_1} c(u) u_t \eta^2 \, dxdt + 2 \iint_{Q_1} c(u) u \eta \eta_t \, dxdt.
\end{align*}
The first term on the right-hand side can be written as
$$
 \iint_{Q_1} c(u)  u_t \eta^2 \, dxdt  = \iint_{Q_1} \partial_t G(u) \eta^2 \, dxdt =  - 2\iint_{Q_1} G(u) \eta \eta_t \, dxdt.
$$
Therefore, using \eqref{eq:G} and the fact that $|c(u)u|\leq Ku^2$,
\begin{equation}\label{eq:bd2}
\left |\iint_{Q_1} \partial_t c(u) u \eta^2 \, dxdt\right |
\leq  3C_0 K \iint_{Q_1} u^2 \, dxdt. 
\end{equation}
Combining \eqref{eq:bd1} and \eqref{eq:bd2}, we see that
\begin{equation}\label{eq:bd3}
\iint_{Q_r}   u_x^2 \, dxdt \leq \iint_{Q_1}   u_x^2 \eta^2\, dxdt \leq 2C_0(2C_0 + 3K) \iint_{Q_1} u^2 \, dxdt. 
\end{equation}

On the other hand, multiplying the equation by $\varphi=c(u)\eta^2$, and integrating over $(-1,1) \times (-1,t)$ with $-r^2<t<0$, we get
$$
\int_{-1}^t \int_{-1}^1 (\partial_s c(u) - u_{xx}) c(u) \eta^2 \, dxds =0.
$$
Integrating by parts in time, and using the properties of $\eta$,
\begin{align*}
\int_{-1}^t\int_{-1}^1 \partial_s c(u)  c(u) \eta^2 \, dxds 
&= \frac{1}{2}\int_{-1}^1 \int_{-1}^t \partial_s (c(u)^2)  \eta^2 \, dsdx\\
&=  \frac{1}{2} \int_{-1}^1 c(u)^2 \eta(\cdot, t)^2\, dx - \int_{-1}^t \int_{-1}^1 c(u)^2 \eta \eta_s\,dxds \\
& \geq \frac{1}{2} \int_{-r}^r c(u)^2 \, dx -  C_0 K^2  \iint_{Q_1} u^2 \, dxds.
\end{align*}
Moreover, integrating by parts in space, using \eqref{eq:bd3}, and the properties of $c$,
\begin{align*}
- \iint_{Q_1} u_{xx}  c(u) \eta^2 \, dxds 
&= \iint_{Q_1} c'(u) u_x^2 \eta^2\, dxds + 2\iint_{Q_1} u_x c(u) \eta \eta_x \, dxds\\
& \leq K\iint_{Q_1} u_x^2 \eta^2\, dxds + C_0 \iint_{Q_1} u_x^2 \eta^2 \, dxds +C_0 \iint_{Q_1} c(u)^2 \,dxds\\
& \leq C_1(K) \iint_{Q_1} u^2 \, dxdt,
\end{align*}
where $C_1(K)>0$ depends only on $K$.
Therefore, for any $-r^2<t<0$, we get
$$
\int_{-r}^r c(u)^2\, dx \leq 2 ( C_1(K) + C_0 K^2) \iint_{Q_1} u^2\, dxds.
$$
Taking the supremum in time, we obtain the result.

\end{proof}

\begin{rem}
The energy estimate also holds for weak solutions to $\partial_t c(u) - u_{xx}=0$, with $c$ satisfying $(A_1)$, following the approximation procedure in Section~\ref{sec:mainstrategy}, and taking the limit. 
\end{rem}

\subsection{Approximating weak solutions} \label{sec:approxweak}
We consider the Cauchy-Dirichlet problem in $Q_1$:
\begin{equation} \label{pb:main2}
\begin{cases}
\partial_t c(u) - u_{xx} =0 & \text{ in } Q_1,\\
u(-1,t) = f(t), \ u(1,t)= g(t) & \text{ for } -1<t \leq 0,\\
c(u(x,-1)) = v_0(x) & \text{ for } -1 \leq x \leq 1. 
\end{cases}
\end{equation}
Let $0<\lambda\leq \Lambda$.  We will assume the following hypothesis on the data:
\begin{enumerate}
\item[$(H_1)$] The function $c$ is given by
$$c(u):=u^+=\max\{u,0\}.$$
\item[$(H_2)$] The functions $f$ and $g$ satisfy
$$
-\Lambda \leq f(t)\leq -\lambda \quad \text{and} \quad
\lambda \leq g(t)\leq \Lambda  \quad \text{for a.e.}~t\in  [-1,0].
$$
\item[$(H_3)$] The initial condition $v_0$ satisfies
$$
 v_0(x)=0 \ \text{ for a.e.}~x\in [-1,0]  \quad \text{ and } \quad 0<v_0(x) \leq \Lambda \ \text{ for a.e.}~x \in (0,1].
$$
\end{enumerate}

We can construct weak solutions to \eqref{pb:main2} by mollifying the data and passing to the limit. 
Indeed, let $\eta : \R \to [0,\infty)$ be a smooth function such that $\supp \eta = [-1,1]$ and $\int_\R\eta \, ds = 1$. For $0<\vep<1$, let
$\eta_\vep (s) = \tfrac{1}{\vep}\eta( \frac{s}{\vep} )$ for  $s\in \R.$
Define the smooth functions:
\begin{align*}
f_\vep(t)  := f\ast \eta_\vep (t) & \text{ and } g_\vep(t) := g\ast \eta_\vep(t)   \text{ for } t\in [-1,0],\\
  v_{0\vep}(x) &:= v_0 \ast \eta_\vep(x) \text{ for } x\in [-1,1].
\end{align*}
We assume that $f$ and $g$ are extended outside $[-1,0]$ , preserving the bounds $-\Lambda \leq f \leq -\lambda$ and $\lambda\leq g \leq \Lambda$, and that $v_0$ is extended by zero outside $[-1,1]$,  so that the convolutions are well defined.
Clearly, $f_\vep$ and $g_\vep$ satisfy $(A_2)$ for $\vep$ small enough. By $(H_3)$ and $\supp \eta_\vep = [-\vep,\vep]$, we have 
$$
v_{0\vep}(x) = \int_{(x-\vep)^+}^{\min\{(x+\vep)^+,\, 1\}} v_0(y) \eta_\vep(x-y) \, dy \quad \text{ for } x \in [-1,1].
$$
Hence, $v_{0\vep}$ satisfies $(A_5)$ with $\gamma_0=-\vep$, i.e.,
$$
v_{0\vep}=0 \  \text{ on  } [-1,-\vep] \quad\text{ and } \quad 0<v_{0\vep}\leq \Lambda \ \text{ on } (-\vep,1],
$$ 
where the upper bound follows from $v_0 \leq \Lambda$ a.e.~and $\int_\R \eta_\vep\, ds=1$.

Let $u_{0\vep}: [-1,1] \to [-\Lambda, \Lambda]$ be a smooth function such that $u_{0\vep}^+ = v_{0\vep}$.  
We also need that
$$
u_{0\vep}(-1) = f_\vep(-1), \quad u_{0\vep}(1) = g_\vep(-1), \quad \text{ and } \quad u_{0\vep}''(-1)=u_{0\vep}''(1)=0.
$$
If this is not the case, we can simply multiply $f_\vep$ and $g_\vep$ by a suitable smooth cut-off function.

By construction, we have $c$, $f_\vep$, $g_\vep$, $v_{0\vep}$, and $u_{0\vep}$ satisfy the assumptions $(A_1)-(A_5)$. By Theorem~\ref{thm:classical}, there exists a unique weak solution $u_\vep$ to \eqref{pb:main2} with data  $f_\vep$, $g_\vep$, and $v_{0\vep}$.

By the energy estimate (Lemma~\ref{lem:energy}),
$$
    \sup_{-r^2<t<0} \int_{-r}^r (u_\vep^+)^2  \, dx + \iint_{Q_{r}} (\partial_x u_\vep)^2 \, dxdt
    \leq C(\lambda, \Lambda),
    $$
    for any $0<r\leq 1/2$. Hence,
    \begin{align*}
    u_\vep \to u & \text{ weakly in } L^2(-1,0; H^1(-1,1)),\\
    u_\vep^+ \to u^+ & \text{ weakly in } C(-1,0; L^2(-1,1)).
    \end{align*}
By uniqueness, and the fact that $f_\vep \to f$, $g_\vep\to g$, and $v_{0\vep}\to v_0$ a.e., we conclude that $u$ is the weak solution to \eqref{pb:main2} with data $f$, $g$, and $v_0$. \medskip

Next, we justify that it is enough to prove our main result for the approximating solutions.

\begin{lem} \label{lem:limit}
If $u_\varepsilon$ satisfies the conclusion of Theorem~\ref{thm:main}, for all $0<\vep<1$, then the same properties also hold for the limit function $u$.
\end{lem}

\begin{proof}
Let $K,L,M,N>0$ and $0<\rho_0,\delta_0<1/2$ be the universal constants given in Theorem~\ref{thm:main}.
For any $0<\vep<1$, assume that $u_\vep$ satisfies the properties $(i)$, $(ii)$, and $(iii)$.
We will show that the limit function $u$ inherits these properties as well.

First, we prove that $u$ satisfies $(i)$. Indeed, by the $L^\infty$ estimate in $(i)$ for $\partial_x u_\vep$, we have $\|\partial_x u_\vep\|_{L^\infty(Q_{1/2})} \leq L$. Since $\partial_ x u_\vep \to u_x$ weakly in $L^2(Q_1)$, then 
$$\|u_x\|_{L^\infty(Q_{1/2})}\leq L.$$
Moreover, by the $C^{1,1/2}$ estimate in $(i)$ for $u_\vep^+$, and Arzel\`a-Ascoli theorem, up to a subsequence, it follows that $u_\vep^+ \to u^+$ uniformly on $\overline{Q_{1/2}}$, and
\begin{equation} \label{eq:holderest}
[u^+]_{C_{x,t}^{1,1/2}(\overline{Q_{1/2}})} \leq M.
\end{equation}

Next, we prove that $u$ satisfies $(ii)$ and $(iii)$, simultaneously. By assumption, the zero-level set of $u_\vep$ in $Q_{1/2}$ is given by the graph of a function 
$\gamma_\vep: [-1/4,0] \to [-1+\delta_0,1-\delta_0]$, with $[\gamma_\vep]_{C^{1/2}([-1/4,0])}\leq K$. Up to a subsequence, 
$\gamma_\vep\to \gamma$ uniformly on $[-1/4,0]$ for some function 
$$\gamma : [-1/4,0] \to [-1+\delta_0,1-\delta_0],\quad  \text{ with  } \quad [\gamma]_{C^{1/2}([-1/4,0])}\leq K.$$

Let $(x_0,t_0)\in Q_{1/2}\cap\{x=\gamma(t)\}$. For $\vep>0$ sufficiently small, we have $Q_{\rho_0/2}(x_0,t_0)\subset Q_{\rho_0}((\gamma_\vep(t_0),t_0))$.
By the non-degeneracy property in $(ii)$ for $\partial_x u_\vep$, and the weak convergence of the spatial derivative in $L^2(Q_1)$, we get
\begin{equation} \label{eq:nondegeneracy}
\partial_x u\geq N \quad \text{ a.e. in } Q_{\rho_0/2}(x_0,t_0).
\end{equation}
It remains to identify the zero-level set of $u$ as
\begin{equation} \label{eq:zerolevelset}
 Q_{1/2}\cap \{u=0\} = \{ (x,t)\in Q_{1/2} : x=\gamma(t)\}.
\end{equation}
Note that $u^+(\gamma(t),t)=0$, since
$$
u^+(\gamma(t),t) \leq \|u_\vep^+-u^+\|_\infty + L \|\gamma_\vep-\gamma\|_\infty\to 0\quad \text{ as } \vep\to 0.
$$
We will characterize the negative and positive phases of $u$ as the domains to the left and to the right, respectively, of the curve $x=\gamma(t)$.
Indeed, let  $E$ be a compact set in $Q_{1/2}\cap \{x<\gamma(t)\}$.
For $\vep>0$ small enough, we have
$E\subset Q_{1/2}\cap\{u_\vep<0\}=Q_{1/2}\cap \{x<\gamma_\vep(t)\}$.
Since $f_\vep\to f$ in $L^2$ and $\gamma_\vep\to \gamma$ uniformly, with $\gamma_\vep\geq -1+\delta_0$, then
$$
u_\vep(x,t)=- \frac{f_\vep(t)}{1+\gamma_\vep(t)}(x+1)+ f_\vep(t) \to - \frac{f(t)}{1+\gamma(t)}(x+1)+ f(t) \quad \text{ in } L^2(E).
$$
Also, $u_\vep \to u$ weakly in $L^2$, so by uniqueness of weak limits, we get
$$
u(x,t)= - \frac{f(t)}{1+\gamma(t)}(x+1)+ f(t) \quad \text{ a.e.~in } E.
$$
Since $E$ is arbitrary, $u<0$ a.e.~in $Q_{1/2}\cap \{x<\gamma(t)\}$.

On the other hand, let $F$ be a compact set in $D:=\{(x,t): \gamma(t)<x<1, \ -1/4<t\leq 0\}$. For $\vep>0$ small enough, we have 
$F\subset \{u_\vep>0\}$. Hence, $u_\vep^+=u_\vep$ in $F$, and uniqueness of weak limits gives $u^+=u$ in $F$.
Since $F$ is arbitrary, $u\geq 0$ in $D$, and
$$
\begin{cases}
    u_t - u_{xx} = 0 & \text{ for } \gamma(t)<x< 1, \ -1/4<t\leq 0,\\
    u(1,t)\geq \lambda>0 & \text{ for a.e. }  -1/4<t\leq 0.
\end{cases}
$$
By the strong maximum principle, $u>0$ in $D$. In particular, $u>0$ in $Q_{1/2}\cap \{x>\gamma(t)\}$.

Therefore, \eqref{eq:zerolevelset} holds, and $u$ satisfies the conclusion of Theorem~\ref{thm:main}.
\end{proof}

\section{Lipschitz regularity in space} 

In this section, we prove the following uniform Lipschitz estimate for the approximating solution $u_\vep$ defined in Section~\ref{sec:approxweak}. For simplicity, we will drop the subindex $\vep$.

\begin{thm} \label{thm:lipschitz}
There exists a universal constant $L>0$ such that
    $$|u_x| \leq L \quad \text{ a.e. in } Q_{1/2}.$$
\end{thm}

\begin{rem}
The estimate holds for all $(x,t) \in Q_{1/2}^+(u)$ since $u$ is smooth in $\{u>0\}$.
\end{rem}

This theorem will be a consequence of the following two lemmas.

\begin{lem}[Local in time maximum principle] \label{lem:locmp}
Let $D$ be a parabolic domain in $\R^{2}$ as in \eqref{eq:domD}, with $t_1=-1$ and $t_2=0$.
Let $f \in L^2(I)$, where $I=(a(-1),b(-1))$.
Let $v$ satisfy
$$
\begin{cases}
v_t - v_{xx} \leq 0 & \text{ on } D,\\
v \leq 0 & \text{ on } \partial_l D,\\
v(x,-1)=f(x) & \text{ for } x\in I.
\end{cases}
$$
Then 
$$
\sup_{D\cap\{t\geq -3/4\}} v \leq | I |^{1/2} \|f^+\|_{L^2(I)}.
$$
\end{lem}

\begin{proof}
For $(x,t)\in \R\times(-1,\infty)$, let 
$$
\Phi(x,t)= \frac{1}{\sqrt{4\pi(t+1)}} \int_I  e^{-\frac{|x-y|^2}{4(t+1)}} f^+(y)\, dy.
$$
Then
$$
\begin{cases}
\Phi_t - \Phi_{xx}=0 & \hbox{in}~\R\times (-1,\infty),\\
\Phi(x,-1)= f^+ & \text{for } x\in \R.
\end{cases}
$$
Note that if $t\geq -3/4$, then $\frac{1}{\sqrt{4\pi(t+1)}}\leq \frac{1}{\sqrt\pi}\leq 1$, and by the Cauchy-Schwarz inequality,
$$
\Phi(x,t) \leq | I | ^{1/2}\|f^+\|_{L^2(I)}.
$$
By construction, we have that $\Phi>0\geq v$ on $\partial_l D$ and $\Phi(x,-1)\geq v(x,-1)$ for a.e.~$x\in I$. By the maximum principle for the heat equation, it follows that $\Phi\geq v$ in $D$. Therefore,
$$
\sup_{D\cap\{t\geq -3/4\}} v \leq | I | ^{1/2}  \|f^+\|_{L^2(I)}.
$$
\end{proof}

Next, we prove that the free boundary is uniformly away from the lateral boundary of $Q_1$. This result will be crucial to obtain universal quantitative estimates.

\begin{lem}[Uniform separation]
\label{lem:unifsep}
There are universal constants $\delta_1, \delta_2\in (0,1/2)$ such that
    $$
    -1 + \delta_1 < \gamma(t) < 1-\delta_2 \quad \text{for all } t \in [-1/2, 0].
    $$
\end{lem}

\begin{proof}
    We will establish the lower bound by constructing an appropriate classical supersolution, and applying the comparison principle. Similarly, we will construct a subsolution to obtain the upper bound.\medskip

    \noindent \textbf{Step 1.} \textit{Lower bound.}  Consider the time-independent function
    $$
    \bu(x)= \Lambda + x (\Lambda + \lambda) \quad \text{ for }  x\in [-1,1].
    $$
    Note that $\bu$ vanishes only at the constant value
    $$
\bar \gamma := -1 + \frac{\lambda}{\lambda+\Lambda}.
    $$
    Since $\bu_t=0$ and $\bu_{xx}=0$, we have that $\bu$ is a classical stationary solution to $\partial_t \bar u^+-\bu_{xx}=0$.
    Moreover, $\bu$ satisfies the following properties:
    \begin{enumerate}[$(i)$]
        \item $\bu >0$ for $x> \bar\gamma$ and $\bu < 0$ for $x< \bar\gamma$;
        \item $\bu(-1)=-\lambda$ and $\bu(1)=2\Lambda +\lambda$;
        \item $\bu^+(x)\geq 0$ for all $x\in [-1,0]$, and $\bu^+(x)\geq \Lambda$ for all $x\in [0,1]$.
    \end{enumerate}
    Then, for $t\in [-1,0]$, we have
    $$
    u(-1,t)\leq -\lambda = \bu(-1,t) \quad \text{and} \quad u(1,t)\leq \Lambda \leq \bu(1,t).
    $$
    Moreover, for $x\in [-1,1]$,
    $$
    u^+(x,-1) \leq \bu^+(x) .
    $$
    By the comparison principle (Theorem~\ref{thm:comparison}), we have $c(\bu)=\bar u^+ \geq u^+=c(u)$~in $Q_1$.
In particular, the positivity set of $u$ is contained in the positivity set of $\bar u$, and thus,
$$
 \gamma(t)\geq \bar\gamma= -1 + \frac{\lambda}{\lambda+\Lambda} \qquad \text{for all } t\in [-1,0].
$$

\medskip

     \noindent \textbf{Step 2.} \textit{Upper bound.} 
     Fix $0<\vep<1$ and $K>0$ to be chosen.
     Define
     $$
    \ur_\vep (t) := 1 - \vep (t+1) \quad \text{ for } t\in [-1,0].
    $$
    Note that $0< 1-\vep \leq \ur_\vep (t)\leq 1$.
    Consider a traveling wave solution to the heat equation,
    $$
    \uu(x,t) = K\big(e^{\vep(x- \ur_\vep(t))}-1\big) \quad \text{ for } (x,t) \in Q_1.
     $$
    We have that $\uu>0$ if $x> \ur_\vep(t)$, $\uu=0$ if $x=\ur_\vep(t)$, and $\uu<0$ if $x< \ur_\vep(t)$. Moreover, 
    $$
      \uu_{xx}(x,t) = \vep^2 K e^{\vep(x- \ur_\vep(t))}>0.
    $$
    Since $\ur_\vep(-1)=1$, then $\uu(x,-1)\leq 0$ for all $x\in [-1,0]$. In particular, for $x\in [-1,0]$,
    $$
    \uu^+(x,-1) =0 \leq u^+(x,-1).
    $$
    Hence, it remains to check the boundary data is ordered, i.e., for $ t\in [-1,0],$
     $$\uu(1,t)=K(e^{\vep^2 (t+1)}-1) \leq \lambda \leq u(1,t),$$
     and 
     $$
     \uu(-1,t)=K\big( e^{\vep(-2+\vep(t+1))}-1\big) \leq -\Lambda \leq u(-1,t).
     $$
    To satisfy these inequalities, it is enough to find $0<\vep<1$ and $K>0$ such that
    \begin{equation} \label{eq:K}
     \frac{\Lambda}{1-e^{\vep^2-2\vep}} \leq K \leq \frac{\lambda}{e^{\vep^2}-1}.
    \end{equation}
 This is possible since $e^{\vep^2}-1\sim \vep^2$ and $1-e^{\vep^2-2\vep}\sim \vep$ as $\vep\to 0$. Hence,
   $$
    \frac{\Lambda}{\lambda} \leq \frac{1-e^{\vep^2-2\vep}}{e^{\vep^2}-1}\sim \frac{1}{\vep} \to \infty\quad \text{ as } \vep \to 0.
   $$
We choose $\vep$ sufficiently small, depending only on $\lambda$ and $\Lambda$, such that the above inequality is satisfied, and then, we choose $K$ satisfying \eqref{eq:K}.
   
   By the comparison principle, it follows that $c(\uu)=\uu^+\leq u^+=c(u)$ in $Q_1$, and thus, 
   $$
    \gamma(t) \leq \ur_\vep(t)= 1-\vep( t+1) \leq 1-\vep/2 \quad \text{for all } t\in [-1/2,0].
    $$    
\end{proof}

We are ready to prove the main theorem of this section.

\begin{proof}[Proof of Theorem~\ref{thm:lipschitz}]
  Since $u$ is an approximating weak solution, by Theorem~\ref{thm:classical} and Section~\ref{sec:ellipticparabolic}, there is a continuous function $\gamma(t)$ such that $\Gamma(u)=\{x=\gamma(t)\}$. Moreover, 
$$
u(x,t) =- \frac{f(t)}{1+\gamma(t)}(x+1) + f(t) \quad \text{ in } Q_1^-(u),
$$
and $u$ satisfies the free boundary problem
$$
\begin{cases}
u_t - u_{xx} = 0 & \text{ in } Q_1^+(u),\\
u_x^+ = u_x^- & \text{ on } \Gamma(u),
\end{cases}
$$
where the free boundary condition is satisfied almost everywhere.
In particular, we have
   $$
   u_x^-(\gamma(t),t)= -\frac{f(t)}{1+\gamma(t)} \quad \text{ for } t\in [-1,0].
    $$
    By Lemma~\ref{lem:unifsep}, there exist universal constants $\delta_1, \delta_2 \in (0, 1/2)$ such that
    $$
    -1+ \delta_1 \leq \gamma(t) \leq 1- \delta_2 \quad\text{for all } t\in [-1/2,0].
    $$
   Since $-\Lambda\leq f(t) \leq -\lambda$, then
   $$
    |u_x^+(\gamma(t),t)| \leq \frac{\Lambda}{\delta_1}  \quad \text{for a.e.}~t\in [-1/2,0].
   $$
   Moreover, note that the line $x=1-\frac{\delta_2}{2}$ stays at a universal positive distance from both the free boundary and the lateral boundary $x=1$. 
   Hence, by interior regularity estimates for the heat equation, we have that
   $$
    \sup_{t\in [-1/2,0]} |u_x\big(1-\tfrac{\delta_2}{2},t\big)| \leq \frac{C}{\delta_2} \|u\|_{L^2(Q_1)}\leq \frac{C \Lambda}{\delta_2}.
   $$
Let $\Gamma_1 = \Gamma(u) \cap \{t\geq -1/2\}$ and $\Gamma_2 =  \{ 1- \tfrac{\delta_2}{2}\} \times [-1/2,0]$. Let $D$ be the parabolic domain  between the free boundary portion $\Gamma_1$ and the segment $\Gamma_2$.
   Set $M=\sup_{\Gamma_1\cup \Gamma_2} u_x$. Note that $M \leq \Lambda\max\{1/\delta_1, C/\delta_2\}$.
   By Lemma~\ref{lem:locmp}, applied to $v=u_x-M$, we get
   $$
    \sup_{{Q_{1/2}^+(u)}} u_x 
    \leq C\big( \|u_x\|_{L^2({Q_{3/4})}} + M\big),
   $$
   where $C>0$ is a universal constant.  By the energy estimate (Lemma~\ref{lem:energy}),
    $$
    \|u_x\|_{L^2(Q_{3/4})} \leq C \|u\|_{L^2(Q_1)}\leq C\Lambda.
    $$
    Therefore, there is a universal constant $L>0$ such that
    $$
 u_x \leq L \quad \text{ a.e. in } Q_{1/2}. 
    $$
    The lower bound follows analogously, taking $M=\inf_{\Gamma_1\cup \Gamma_2} u_x$ and $v=M-u_x$. 
   \end{proof}

   \begin{rem} \label{rem:higherdim}
The one-dimensional assumption is crucial in this proof. Since $u$ is linear (in $x$) in the negative phase, the slope, $u_x^-$, is explicitly determined by the boundary value $f(t)$ and the free boundary $\gamma(t)$. By the transmission condition, this gives a bound for the one-sided spatial derivative from the parabolic phase, $u_x^+$, at the free boundary.

In higher dimensions, the negative phase is an unknown domain, and the
elliptic solution depends on the geometry of that domain in a genuinely
nonlocal way. Obtaining analogous control of its normal derivative at the
free boundary from bounded boundary data is a challenging problem that is
beyond the scope of the present work.
\end{rem}

\subsection{Improved time regularity in unsaturated domain}
In the parabolic side, the Lipschitz regularity in space implies the $C^{1/2}$ regularity in time up to the free boundary. 

\begin{thm}[$C^{1/2}$-time regularity] \label{thm:timereg}

 There is a universal constant $C>0$ such that
    $$
    |u(x,t_1)-u(x,t_2)| \leq C |t_1-t_2|^{1/2},     
    $$
    for all $(x,t_1), (x,t_2) \in \overline{Q_{1/2}^+(u)}.$
\end{thm}

\begin{proof}
Assume that $u$ is a weak solution to
$$\partial_t c(u) - u_{xx} =0 \ \text{ in } Q_1, \quad \text{ with } \  c(u)=u^+.$$
Let $w:=c(u)=u^+$. By Theorem~\ref{thm:lipschitz} (rescaled), we have $\|u_x\|_{L^\infty(Q_{3/4})}\leq L$, thus
    $$w_x =u_x \chi_{\{u>0\}} \quad \text{ and } \ |w_x|\leq L \quad \text{ a.e.~in } Q_{3/4}.$$
    Hence, $w(\cdot, t)$ is Lipschitz in space, uniformly in $t$, for almost every $t$. Since $w$ is continuous, the Lipschitz estimate holds for every $t$.

    Let  $\eta : \R \to [0,\infty)$ be a smooth function such that $\supp \eta \subseteq [-1/2,1/2]$ and $\int_\R\eta \, dy = 1$.  
    Set $C_0:= \|\partial_y \eta\|_{L^1(\R)}$.
    For $r\in (0, 1/4)$, consider $\eta_{r}(y):=r^{-1} \eta(\frac{y}{r})$. Then
    $$
    \supp \eta_r \subseteq [-r/2, r/2], \quad \int_\R \eta_r \, dy=1, \quad \text{ and } \quad \|\partial_y \eta_r\|_{L^1(\R)}= C_0r^{-1}.
    $$
    For $(x,t)\in Q_{1/2}$, define
    $$
    \phi_r(x,t):= (w\ast \eta_{r})(x,t) = \int_\R w(y,t) \eta_{r}(y-x)\, dy.
    $$
    Fix $x\in (-1/2,1/2)$. For every $t \in (-1/4, 0]$, we have
    \begin{equation} \label{eq:step1}
    |w(x,t)-\phi_r(x,t)| \leq \int_{\R} |w(x,t)-w(y,t)| \eta_{r}(y-x)\, dy\leq Lr.
    \end{equation}
    For any $\psi\in C^\infty_c((-1/4,0))$, testing the weak equation, with $\varphi(y,t)= \eta_r(y-x) \psi(t)$, we get
    $$
    \iint_{Q_{1}}(- w  \varphi_t + u_y \varphi_y)\, dydt=0,
    $$
    or equivalently,
    $$
    \int_{-1/4}^0 \Big(\int_{-3/4}^{3/4} w(y,t) \eta_r(y-x)dy \Big) \psi'(t) \, dt =   \int_{-1/4}^0 \Big(\int_{-3/4}^{3/4}  u_y(y,t) \partial_y \eta_r(y-x)\,dy\Big) \psi(t) \, dt.
    $$
    Since $\psi$ is arbitrary, then
    $$
    \partial_t \phi_r(x,t) = - \int_{-3/4}^{3/4} u_y(y,t) \partial_y \eta_r(y-x)dy,
    $$
    in the distributional sense, and 
    \begin{equation} \label{eq:step2}
    |\partial_t \phi_r(x,t)|\leq \int_{-3/4}^{3/4}  |u_y(y,t)| |\partial_y \eta_r(y-x)| \,dy \leq L \|\partial_ y\eta_r\|_{L^1(\R)} = LC_0 r^{-1}.
    \end{equation} 

Therefore, for any $t_1, t_2\in (-1/4,0)$, with $t_1\neq t_2$, setting $r:=|t_1-t_2|^{1/2}/2<1/4$, and using \eqref{eq:step1} and \eqref{eq:step2}, it follows that
    \begin{align*}
        |w(x,t_1)- w(x,t_2)| 
        & \leq  |w(x,t_1)- \phi_r(x,t_1)| + |\phi_r(x,t_1)-\phi_r(x,t_2)| +  |\phi_r(x,t_2)- w(x,t_2)|\\
        & \leq Lr + L C_0|t_1-t_2| r^{-1}+Lr\\
        &\leq C|t_1-t_2|^{1/2},
    \end{align*}
    where $C>0$ is a universal constant. By continuity of $w$, the estimate extends to the closure $\overline{Q_{1/2}}$.
    Finally, since $w=u$ on $\overline{Q_{1/2}^+(u)}$, we obtain the result.
\end{proof}

As an immediate consequence of Theorems~\ref{thm:lipschitz} and \ref{thm:timereg} we obtain the $C^{1,1/2}$ regularity in space and time in the unsaturated region.

\begin{cor}[$C^{1,1/2}$ regularity] \label{cor:regu}
There is a universal constant $M>0$ such that
    $$
    |u(x_1,t_1)-u(x_2,t_2)| \leq M \big( |x_1-x_2|^2 + |t_1-t_2| \big)^{1/2},
    $$
    for all $(x_1,t_1), (x_2, t_2) \in \overline{Q_{1/2}^+(u)}$.
\end{cor}

\section{Non-degeneracy near the free boundary}

We will prove the non-degeneracy near the free boundary for the approximating weak solutions, which we continue calling $u$.
A key assumption is that
$$
u(1,t) \geq \lambda, \quad \text{for all } t\in [-1,0].
$$
Let $(x_0,t_0)\in \Gamma_{1/2}(u)=Q_{1/2}\cap\{u=0\}$ be a free boundary point. We need to show that
$$
u_x\geq N \quad \text{ a.e.~in } Q_{\rho_0}(x_0,t_0),
$$
for some universal constants $0<\rho_0<1/2$ and $N>0$.
Without loss of generality, we may assume that $(x_0,t_0)=(0,0)$.
Indeed, let $r_0:=1-x_0$ be the distance from $x_0$ to the right boundary $x=1$. Since $-1/2<x_0< 1/2$, then $1/2< r_0 < 3/2$. 
Consider the rescaled function $v(x,t)=u(r_0 x+x_0,t/4+t_0)$ for $(x,t)\in Q_1$. 
If $x_0-r_0<-1$, then we extend $u$ linearly to the left of $x=-1$, in the natural way.
Note that $v$ satisfies $v(0,0)=0$ and 
\begin{equation} \label{eq:nondegen}
v(1,t)=u(1,t/4+t_0) \geq \lambda  \quad \text{ for all } t \in [-1,0].
\end{equation}
Moreover, $v$ is a weak solution to $\partial_t \widetilde c(v)- v_{xx}=0$ in $Q_1$, where $\widetilde c(v)=4 r_0^2 v^+$. Since $1< 4 r_0^2 < 9$, the estimates will remain uniform in the coefficient $4r_0^2$.
Finally, if $v_x\geq N$ in $Q_{\rho_0}$, then $u_x\geq 2N/3$ in $Q_{\rho_0/2}(x_0,t_0)$.
Hence, it is enough to prove the non-degeneracy estimate for $(x_0,t_0)=(0,0)$. We prove this in the next result.

\begin{thm}\label{thm:monotonicity}
 There exist universal constants $0<\rho_0<1/2$ and $N>0$ such that
    $$
 u_x \geq N \quad \text{ a.e.~in } Q_{\rho_0}.
    $$
\end{thm}

\begin{proof}
Recall that for $-1< x < \gamma(t)$ and $t\in [-1,0]$, we have
        $$
        u(x,t)= -  \frac{f(t)}{1+\gamma(t)}(x+1)+f(t).
        $$
Moreover, using $f(t) \leq -\lambda$ and $0<1+\gamma(t)\leq 2$, we get
\begin{equation} \label{eq:bound1}
u_x^-(x,t) = -\frac{f(t)}{1+\gamma(t)}\geq \frac{\lambda}{2}>0 \quad \text{ in } Q_1^-(u).
\end{equation}
Hence, it remains to prove the estimate for $u_x$ in $Q_{\rho_0}^+(u)$ for some $\rho_0$ small universal. 
Note that, by Theorem~\ref{thm:lipschitz} (rescaled), we have 
\begin{equation} \label{eq:bound2}
 u_x \geq - L \quad \text{ a.e. in } Q_1.
\end{equation}
We will improve the lower bound in several steps.\medskip

\noindent \textbf{Step 1.} \textit{Selection of a level set of $u$.} 
For $0<a<\lambda$, 
consider the level set 
$$\Gamma_a:=\{(x,t)\in Q_1^+(u) : u(x,t)=a\}.$$ 
Since $u$ is continuous on ${Q_1^+(u)}$, $u(0,0)=0$, and $u(1,0)\geq \lambda$, we have that $\Gamma_a \neq \emptyset$. By Sard's theorem, $\Gamma_a$ contains a finite number of smooth curves for a.e.~$a$. By the strong maximum principle, none of these curves can be homeomorphic to a circle.
Moreover, $\Gamma_a$ must contain a curve that disconnects the free boundary $\Gamma(u)$ from the lateral boundary $\{x=1\}$. Among all such curves, choose the leftmost one and call it $x=s_a(t)$. Then $u<a$ in the region between the free boundary and $s_a(t)$, and thus,
\begin{equation} \label{eq:monsa}
u_x(s_a(t),t) \geq 0 \quad \text{ for all } t\in[-1,0].
\end{equation}

\noindent\textbf{Step 2.} \textit{Construction of barriers.}
Define the domain
$$A := \{ (x,t)\in Q_1 : \gamma(t) < x < s_a(t) \text{ and } -1<t\leq 0\};
$$
see Figure~\ref{fig2}. Let $v$ be the solution to the Cauchy-Dirichlet problem
$$
\begin{cases}
    v_t - v_{xx}=0 & \text{ in } A,\\
    v = u_x & \text{ on } \partial_l A,\\
    v = 0 & \text{ on } \partial_b A.
\end{cases}
$$
By \eqref{eq:bound1}, \eqref{eq:monsa}, and the maximum principle, we have that
\begin{align} \nonumber 
    & v \geq 0  \text{ on } A,\\ \label{eq:boundv}
    & v \geq \lambda/2  \text{ on } \partial_l A\cap \{u=0\},\\ \nonumber
    & v \geq 0  \text{ on } \partial_l A \cap \{u=a\},
\end{align}
where the boundary inequalities are satisfied for almost every time. 

\begin{figure}[h]
\centering
\begin{tikzpicture}[scale=4] 

\draw[thick] (-1,-1) rectangle (1,0);

\draw[thick] (-1, 0) -- (1, 0);             
\draw[dashed] (-1,-0.2) -- (0.4,-0.2);       
\draw[thick] (-1,-1) -- (1,-1);             

\node[left] at (-1,  0) {\scriptsize $0$};
\node[left] at (-1,-0.2) {\scriptsize $-\rho^2 a^2$};
\node[left] at (-1, -1) {\scriptsize $-1$};
\node[below] at (-1, -1) {\scriptsize $-1$};
\node[below] at ( 1, -1) {\scriptsize $1$};


\begin{pgfonlayer}{background}
  \fill[gray!10]
    plot[domain=0:1, samples=100, smooth] 
      ({0.08*sin(deg(6*pi*\x))}, {-\x})
    --
    plot[domain=1:0, samples=100, smooth] 
      ({0.57 + 0.04*sin(deg(4*pi*\x + 0.5))}, {-\x})
    -- cycle;
    \end{pgfonlayer}


\draw[thick, domain=0:1, smooth, samples=100, variable=\t] 
  plot({0.08*sin(deg(6*pi*\t))}, {-\t});

\draw[thick, domain=0:1, smooth, samples=100, variable=\t] 
  plot({0.57 + 0.04*sin(deg(4*pi*\t + 0.5))}, {-\t});

\draw[dashed] (0.4, 0) -- (0.4, -0.2);

\filldraw (0,0) circle (0.015);
\filldraw (0.4,0) circle (0.015);

\node[above] at (0,0) {\scriptsize $(0,0)$};
\node[above] at (0.4,0) {\scriptsize $(\rho a,0)$};

\draw[->] (1,-1) -- (1.1,-1) node[below] {\scriptsize $x$};
\draw[->] (-1,0) -- (-1,0.1) node[left] {\scriptsize $t$};

\node at (0.25, -0.1) {\scriptsize $Q_{\rho a}^+(u)$};
\node at (0.3, -0.8) {\scriptsize\normalsize $A$};
\node at (0.3, -0.5) {\scriptsize $v$};
\node at (-0.25, -0.6) {\scriptsize $x=\gamma(t)$};
\node at (0.78, -0.78) {\scriptsize $x=s_a(t)$};
\end{tikzpicture}
\caption{The domain $A$.}
\label{fig2}
\end{figure}

We will show that there is a universal constant $c_0>0$ such that
\begin{equation} \label{eq:infv}
\inf_{Q_{\rho a}^+(u)} v \geq c_0\lambda.
\end{equation}
Indeed, let 
$$
w:= \frac{\lambda}{2a}(a-u) - \phi \quad \text{ in } A,
$$
where $\phi$ is the solution to the problem
$$
\begin{cases}
    \phi_t - \phi_{xx} =0 & \hbox{ in } (-1,1)\times (-1,0],\\
    \phi(-1,t)=\phi(1,t)=0 & \hbox{ for } t\in (-1,0],\\
    \phi(x,-1)=\frac{\lambda}{2} & \hbox{ for } x\in (-1,1).
\end{cases}
$$
By the strong maximum principle, there is a universal constant $0<\kappa<1$ such that
\begin{equation} \label{eq:boundpsi}
    0< \phi \leq \frac{\kappa \lambda}{2} \quad \text{ in } (-1,1)\times[-1/2,0].
\end{equation}
Then $w$ satisfies 
$$
\begin{cases}
    w_t - w_{xx}=0 & \hbox{ in } A,\\
        w= \frac{\lambda}{2} - \phi \leq \frac{\lambda}{2}  & \hbox{ on } \partial_l A\cap\{u=0\},\\
    w=-\phi\leq 0 & \hbox{ on } \partial_l A\cap\{u=a\},\\
    w= \frac{\lambda}{2a}(a-u) - \frac{\lambda}{2}\leq 0 & \hbox{ on } \partial_b A.
\end{cases}
$$
By \eqref{eq:boundv}, we have $v\geq w$ on $\partial_p A$. By the comparison principle, $v\geq w$ in $A$.

Now, by Corollary~\ref{cor:regu}, for any $\rho>0$, with $\rho a\leq 1/2$, since $u(0,0)=0$, we have
$$
u \leq M \rho a \quad \text{ in }~Q_{\rho a}^+(u).
$$
Hence, given $0<\theta<1/2$, choosing  $M\rho \leq \theta$, we get
\begin{equation} \label{eq:away}
    u\leq \theta a \quad  \text{in}~Q_{\rho a}^+(u).
\end{equation}
Since $Q_{\rho a}^+(u)\subset A \cap \{t\geq -1/2\}$, using \eqref{eq:boundpsi}, it follows that
$$
v\geq w \geq \frac{(1-\theta)\lambda}{2}  - \frac{\kappa\lambda}{2} = \frac{1-\kappa}{4}\lambda \quad \text{ in } Q_{\rho a}^+(u),
$$
taking $\theta=\frac{1-\kappa}{2}$. Therefore, we obtain \eqref{eq:infv}.

\medskip
\noindent\textbf{Step 3.} \textit{Improved lower bound for $u_x$.} Let $v$ be as in Step 2. Consider
$$
u_x = v - (v-u_x).
$$
It remains to show that the difference $v-u_x$ can be made arbitrarily small, choosing a universal $a$ sufficiently small. To see this, we will prove that there exist universal $c_1>0$, $a_1\in (0,\lambda)$, and $t_1 \in (-1,-3/4]$ such that 
\begin{equation} \label{eq:small}
    s_{a_1}(t_1)- \gamma(t_1) \leq c_1 a_1.
\end{equation}
Indeed, for $\sigma>0$, let $p_\sigma(t)=1+\sigma(t+1)(t+3/4)$.
Define the domain
$$
D_\sigma :=\{(x,t): p_\sigma(t)<x<1,\ -1<t<-3/4\}.
$$
By Lemma~\ref{lem:unifsep}, recall that  $-1+\delta_1 < \gamma(t)<1-\delta_2$.
Hence, for small $\sigma$, the domain $D_\sigma$ is contained in the strip $\{ 1-\delta_2 \leq x \leq 1\}$, so it does not intersect the free boundary.
We then increase $\sigma$ until the left boundary $x=p_\sigma(t)$ first touches
the free boundary $x=\gamma(t)$. Denote the first touching parameter by $\sigma_1$ and the
corresponding touching point by $(x_1,t_1)$, so
$$
x_1=p_{\sigma_1}(t_1)=\gamma(t_1).
$$
By construction, $D_{\sigma_1}\subset\{u>0\}.$ Moreover, $\sigma_1$ is universally bounded, and $(x_1,t_1)$ is universally away from the corner points $(1,-1)$ and $(1,-3/4)$.

Let $\psi$ be the solution to the Cauchy-Dirichlet problem
\[
\begin{cases}
    \psi_t - \psi_{xx} = 0 & \text{in}~D_{\sigma_1},\\
    \psi(1, t) = \lambda & \text{for}~ t \in (-1,-3/4], \\
    \psi(x, t) = 0 & \text{for } x=p_{\sigma_1}(t),  \ t \in [-1,-3/4].
\end{cases}
\]
By the Hopf Lemma, we have that $c_1^{-1} := \psi_x(x_1, t_1) > 0$.
Since $D_{\sigma_1} \subset \{u>0\}$ and $u(1, t) \geq \lambda$ for $t \in [-1,0]$, then $\psi \leq u$ on $\partial_p D_{\sigma_1}$. By the maximum principle for $u - \psi$ in $D_{\sigma_1}$, it follows that
$\psi \leq u$ on $D_{\sigma_1}$. Moreover, $u(x_1, t_1) = \psi(x_1, t_1) = 0$. Therefore,
\[
u(x, t_1) \geq \psi(x, t_1) \geq c_1^{-1} (x - x_1),
\]
for all $x \in [x_1, x_1 + \varepsilon_1)$, where $\varepsilon_1 \in (0, 1 - x_1)$ is universal. Let $a_1 \in (0, \lambda)$ be such that $s_{a_1}(t_1) < x_1 + \varepsilon_1$.
Then plugging in $x=s_{a_1}(t_1)$ into the above estimate, we get
$$
a_1 = u(s_{a_1}(t_1), t_1) \geq c_1^{-1} (s_{a_1}(t_1)-\gamma(t_1)),
$$
where we used that $x_1= \gamma(t_1)$. Hence, we proved \eqref{eq:small}.

Let $I_1=(\gamma(t_1), s_{a_1}(t_1))$. 
Then $|I_1|\leq c_1 a_1.$
Let $\Phi$ be the convolution of the heat kernel (centered at $t_1$) and the function $L \mathbbm{1}_{I_{1}}$.
Then $\Phi$ satisfies
$$
\begin{cases}
    \Phi_t - \Phi_{xx}=0 & \text{ in } \R\times (t_1,\infty),\\
    \Phi(x,t_1) = L \mathbbm{1}_{I_{1}} & \text{for}~x\in \R.
\end{cases}
$$
For all $x\in \R$ and $t>t_1$, we have
$$
\Phi(x,t) = \frac{L}{\sqrt{4\pi (t-t_1)}} \int_{I_1} e^{-\frac{|x-y|^2}{4(t-t_1)}} \, dy \leq L |I_1|\leq L c_1a_1, 
$$
Then $v-u_x$ satisfies the heat equation in $A$ and $v-u_x \leq \Phi$ a.e.~on $\partial_p A$ by \eqref{eq:bound2} and \eqref{eq:boundv}. By the maximum principle, it follows that 
$v-u_x \leq \Phi$ a.e.~ in $A$. In particular, for a.e.~$(x,t)\in A$ with $t\in (-1/2,0]$, we have
$$
(v-u_x)(x,t) \leq \Phi(x,t) \leq  L c_1a_1.
$$
By \eqref{eq:infv}, and choosing $a_1 < \min\big\{\lambda,\frac{c_0\lambda}{2Lc_1}\big\}$, we get
$$
 u_x \geq N \quad \text{ in } Q_{\rho_0}^+(u),
$$
where $\rho_0= \rho a_1$ and $N= \frac{c_0\lambda}{2}$.
\end{proof}

\section{Regularity of the free boundary}

Recall from Section~\ref{sec:ellipticparabolic} that there is a continuous function $\gamma(t)$ such that
$$
\Gamma(u)= \{ (x,t) \in Q_{1} : x= \gamma(t) \},
$$
where $u$ is an approximating weak solution.
We will show that $\gamma\in C^{1/2}$ with a universal estimate, and that this regularity is optimal under the mild assumptions $(H_2)$ and $(H_3)$.

\subsection{Universal $C^{1/2}$ estimate} \label{sec:hodograph}
Our main result in this section is the following:
\begin{thm}\label{thm:interface}
 There is a universal constant $K>0$ such that
    $$
    |\gamma(t_1)-\gamma(t_2)| \leq K |t_1-t_2|^{1/2} \quad \text{for all}~t_1,t_2 \in [-1/2, 0].
    $$
\end{thm}

To prove this theorem, we will transform the free boundary problem into a fixed transmission problem, by means of a hodograph transform that flattens all level sets of $u$, and in particular, the free boundary. The estimate for $\gamma(t)$ will then follow from the boundary regularity estimates for solutions to the transmission problem.\medskip

Without loss of generality, we may assume that $u(0,0)=0$.
By Theorem~\ref{thm:monotonicity}, there is a universal cylinder $Q_{\rho_0}$, where $u$ is monotone in space. 
Consider the change of variables
\begin{equation}\label{eq:hodograph1}
y=u(x,t) \quad \text{for }  (x,t)\in \overline{Q_{\rho_0}^+(u)}.
\end{equation}
Take $\rho_0/2 \leq a_0\leq \rho_0$ such that the curve $s_{a_0}(t)$ in $\{u=a_0\}$, defined as in the proof of Theorem~\ref{thm:monotonicity}, is contained in $Q_{\rho_0}^+(u)$.
If $\gamma(t) \leq x \leq s_{a_0}(t)$ for some $t\in [-\rho_0^2, 0]$, then $0 \leq y \leq a_0$. Hence, there is some $\sigma_0>0$, depending on $\rho_0$, such that 
$\overline{ Q_{\sigma_0}^+} \subseteq [0, a_0]\times [-\rho_0^2, 0].$

We define the hodograph transform of $u$ as the real-valued function $h$ satisfying
\begin{equation}\label{eq:hodograph2}
h(y,t) = x 	\quad \hbox{for } (y,t)\in \overline{Q_{\sigma_0}^+}.
\end{equation} 
Note that since $u$ is smooth in $Q_1^+(u)$, and continuous up to the free boundary, then 
$$h \in C^\infty(Q_{\sigma_0}^+)\cap C(\overline{Q_{\sigma_0}^+}).$$
Differentiating the equations \eqref{eq:hodograph1} and \eqref{eq:hodograph2}  with respect to $x$ and $t$, we get
$$
 u_x = \frac{1}{h_y},\quad  u_{xx} = - \frac{h_{yy}}{h_y^3}, \quad u_t = - \frac{h_t}{h_y}.
$$
By definition, $h$ is bounded and
$h(0,t)= \gamma(t)$ for all $t \in [-\sigma_0^2,0]$. Moreover, $h$ satisfies
\begin{equation}\label{eq:eqh}
    h_t - \Big(\frac{1}{h_y^2}\Big) h_{yy}=0 \quad \hbox{in}~Q_{\sigma_0}^+,
    \end{equation}
    in the classical sense.
Note that this is a uniformly parabolic equation since 
\begin{equation} \label{eq:unifpar}
0<N \leq \frac{1}{h_y} \leq L \quad \text{ in }   Q_{\sigma_0}^+. 
\end{equation}

\begin{proof}[Proof of Theorem~\ref{thm:interface}]
Let $h$ be the hodograph transform of $u$ defined above on $\overline{Q_{\sigma_0}^+}$. After rescaling, we may assume without loss of generality that $\sigma_0=1$.\medskip

We need to prove that $\gamma=h(0,\cdot)\in C^{1/2}([-1/2,0])$, and
\begin{equation} 
|\gamma(t_1)-\gamma(t_2)|\leq K |t_1-t_2|^{1/2} \quad \text{ for all } t_1, t_2 \in [-1/2,0],
\end{equation}
where $K>0$ is a universal constant.

Indeed, let $w=h_y$. By \eqref{eq:unifpar}, we have 
\begin{equation} \label{eq:boundw}
L^{-1} \leq w\leq N^{-1} \quad \text{ in } Q_1^+.
\end{equation}
 Differentiating \eqref{eq:eqh} with respect to $y$, we see that $w$ satisfies the divergence form equation
 \begin{equation} \label{eq:divform}
w_t - \Big(\frac{1}{w^2}w_{y}\Big)_y=0 \quad \text{in } Q_1^+.
\end{equation}
By parabolic regularity theory, there is a universal $\alpha\in (0,1)$ and a constant $C>0$ such that  $w\in C^{0,\alpha}_{\loc}(Q_1^+)$, and for any compact subset $Q'$ in $Q_1^+$,  
$$
d^{\alpha} [w]_{C^{0,\alpha}({Q'})}\leq C \|w\|_{L^\infty(Q_1^+)} \leq C N^{-1},
$$
where $d$ is the parabolic distance between $Q'$ and $\partial_p Q_{1}^+$.
By \eqref{eq:boundw} and the above estimate, it follows that
$w^{-2}  \in C^{0,\alpha}_{\loc}(Q_1^+),$
with  
$$
d^\alpha [w^{-2}]_{C^{0,\alpha}({Q'})} \leq 2L^{4} N^{-1} d^\alpha [w]_{C^{0,\alpha}({Q'})} \leq 2CL^{4} N^{-2}.
$$
By the Schauder theory for \eqref{eq:divform}, we have $w\in C^{1,\alpha}_{\loc}(Q_1^+)$, and 
\begin{equation}\label{eq:estwy}
\|w\|_{L^\infty(Q')} + d \|w_y\|_{L^\infty(Q')}  + d^{1+\alpha} \|w_y\|_{C^{0,\alpha}({Q'})} \leq C_1, 
\end{equation}
where $C_1>0$ depends only on $\alpha$, $N$, and $L$. In particular, using that $w_y=h_{yy}$, we get
\begin{equation} \label{eq:esthyy}
  \|h_{yy}\|_{L^\infty(Q')} \leq C_1 d^{-1}.
\end{equation}
Let $t_1, t_2 \in [-1/2,0]$ with $t_1<t_2$. Set $\tau=\sqrt{t_2-t_1}\in \big(0,\frac{\sqrt{2}}{2} \big]$.
By the mean value theorem,
\begin{align} \nonumber
    |h(0,t_2)-h(0,t_1)| & 
    \leq |h(0,t_2)-h({\tau},t_2)| 
    + |h(\tau, t_2) - h({\tau}, t_1)|
    + |h({\tau}, t_1)-h(0,t_1)|\\ \nonumber
    & \leq \sup_{y \in [0, \tau]} |h_y(y, t_2)| {\tau} + \sup_{t\in [t_1,t_2]} |h_t({\tau},t)| \tau^2 + \sup_{y \in [0, \tau]} |h_y(y, t_1)| {\tau}\\ \label{eq:est1}
    & \leq 2N^{-1} {\tau} + \sup_{t\in [t_1,t_2]} |h_t({\tau},t)| \tau^2.
\end{align}
We need to control the second term on the right. Take $Q'=[\tfrac{\tau}{16}, \tfrac{5\tau}{4}]\times [t_1,t_2]$, which is a compact set in $Q_1^+$ such that $d\geq \frac{\tau}{16}$.
Using the equation satisfied by $h$ in \eqref{eq:eqh}, the bounds for $h_y$ in \eqref{eq:unifpar}, and the interior estimate in \eqref{eq:esthyy}, we get
\begin{align} \label{eq:est2}
     \sup_{t\in [t_1,t_2]} |h_t({\tau},t)| 
    & = \sup_{t\in [t_1,t_2]}  \frac{|h_{yy}({\tau},t)|}{h_y^2({\tau},t)} \leq L^2 \|h_{yy}\|_{L^\infty(Q')} \leq  L^2 C_1 d^{-1} \leq  16 L^2 C_1 \tau^{-1}.
\end{align}
Since $\gamma(t)=h(0,t)$, combining \eqref{eq:est1} and \eqref{eq:est2}, we obtain
$$
 |\gamma(t_1)-\gamma(t_2)| \leq 2(N^{-1}+8L^2C_1) |t_1-t_2|^{1/2}.
$$
\end{proof}

\begin{rem}
If the boundary data is smoother, we expect to get higher regularity of the free boundary from this purely local hodograph argument. To the best of our knowledge, this method is new in the context of elliptic-parabolic free boundary problems.
\end{rem}

\subsection{Optimality} 
It remains to see that the $C^{1/2}$ regularity of the free boundary is optimal.

 \begin{thm}[Optimal regularity] \label{thm:optimal}
There exists a pair of functions $(u,\gamma)$ that satisfy
\begin{equation} \label{eq:FBP2}
\begin{cases}
u_{xx} = 0 & \text{ for } x_0<x<\gamma(t), \ 0<t\leq T, \\
 u_t - u_{xx} = 0 & \text{ for } \gamma(t) < x < x_1, \ 0<t\leq T,\\
u(\gamma(t),t)=0 & \text{ for } 0<t\leq T,\\
u_x^+ (\gamma(t), t) = u_x^- (\gamma(t),t) & \text{ for } 0<t\leq T,\\
\end{cases}
\end{equation}
in the classical sense, for some $0<T \leq 1$ and $x_0<x_1$. Moreover,
 $$
K_1 \sqrt{t}\leq \gamma(t) \leq K_2 \sqrt{t} \quad \text{ for all } t\in [0,T],
 $$ 
where $K_1< K_2$ are universal constants.
 \end{thm}

 \begin{rem}
It can be deduced from the proof that the function $m(t):=u_x^+(\gamma(t),t)$ has a jump discontinuity at $t=0$. Since $u(x,t)=m(t)(x-\gamma(t))$, for $x\leq \gamma(t),$ then $u$ is discontinuous at $t=0$ in the negative phase.
 \end{rem}

We need the following maximum principle for the Cauchy-Neumann problem; e.g., see \cite[Theorem 7.1]{L}. 
For convenience, we now consider cylinders forward in time:
$$
Q:=(0,1)\times(0,1].
$$ 

\begin{lem}[Maximum principle] \label{lem:MPmixed}
Let $w: \overline{Q}\to \R$ satisfy
\begin{equation*}
\begin{cases}
w_t - a(y,t) w_{yy} \leq 0 & \text{ in } Q,\\
w_y \geq 0 & \text{ on } \partial_l Q,\\
w(y,0) \leq 0 & \text{ on } \partial_b Q,
\end{cases}
\end{equation*}
where $0<L_1 \leq a(y,t) \leq L_2$ for all $(y,t)\in Q$.
Then 
$w\leq 0$ on $\overline{Q}$.
\end{lem}

We also need the following auxiliary function.

\begin{lem}[Barrier] \label{lem:psi}
Let $\kappa>0$. There exists a function $\Psi : [0,\infty) \times [0,\infty)\to \R$ satisfying
 $$
 \begin{cases}
 \Psi_t - \kappa \Psi_{yy}=0 & \text{ for } y>0,\ t>0,\\
 \Psi_ y(0,t)=1 & \text{ for } t>0,\\
 \Psi(y,0)= 0 & \text{ for } y\geq 0,
 \end{cases}
 $$
 such that  $\Psi_{yy}\leq 0$ on $(0,\infty) \times (0,\infty).$ Moreover,
 $$
 \Psi(0, t) = -  \frac{2}{\sqrt{\pi} } \sqrt{\kappa t} \quad \text{ for all } t>0.
 $$
\end{lem}

\begin{proof}
Let $\Phi(y,t)$ be the convolution of the heat kernel with the function $|y|$.
Then $\Phi$ satisfies
$$
 \begin{cases}
\Phi_t -  \Phi_{yy}=0 & \text{ for } y\in \R,\ t>0,\\
 \Phi(y,0)= |y| & \text{ for } y\in \R.
 \end{cases}
$$
Since the initial data $|y|$ is even, $\Phi(\cdot,t)$ is even, and thus,
$\Phi_y(0,t)=0$ for all $t>0$. Since $|y|$ is convex, convolution with the
heat kernel preserves convexity, and therefore
$$
\Phi_{yy}(y,t)\geq 0
\qquad\text{for all } y\in\R,\ t>0.
$$
Moreover, for all $t>0$, a simple computation shows that
$$
\Phi(0,t) = \frac{2}{\sqrt\pi}\sqrt t.
$$

Given $\kappa>0$, we define
$$
\Psi(y,t) := y - \Phi(y, \kappa t), \quad \text{ for } y\geq 0, \ t\geq 0.
$$
From the properties of $\Phi$, it is easy to check that $\Psi$ satisfies the conditions in the statement.
\end{proof}

Theorem~\ref{thm:optimal} will follow from the next key theorem.
 
 \begin{thm}\label{keythm}
There exist $\delta>0$ small universal and a classical solution to
 \begin{equation} \label{eq:optimalsol}
 \begin{cases}
 v_t - \frac{1}{v_y^2} v_{yy} = 0 & \text{ for } 0<y<1, \ 0<t\leq 1,\\
 v_y(0,t)= 1+ \delta & \text{ for } 0 < t \leq 1,\\
 v_y (1, t) = 1+ \delta \Psi_y(1,t) & \text{ for } 0 < t \leq 1,\\
v(y,0)=y & \text{ for } 0\leq y \leq 1,
 \end{cases}
 \end{equation}
 where $\Psi$ is given in Lemma~\ref{lem:psi} with $\kappa=4$, such that
 $$
 v(0,t) \leq - \frac{4\delta}{\sqrt{\pi}} \sqrt{t} \quad \text{ for all } t \in [0,1].
 $$
 \end{thm}
 
 \begin{proof}
 We divide the proof into two steps:\medskip
 
 \noindent \textbf{Step 1.} \textit{Existence.} Fix $\delta >0$ to be chosen small. We will show that there is a classical solution to \eqref{eq:optimalsol} using a fixed point argument. 
 Writing $v=y + \delta w$, it suffices to solve
 \begin{equation} \label{eq:fixedpoint}
 \begin{cases}
 w_t - \frac{1}{(1+\delta w_y)^2} w_{yy} = 0 & \text{ for } 0<y<1, \ 0<t\leq 1,\\
w_y(0,t)= 1 & \text{ for } 0 < t \leq 1,\\
 w_y (1, t) =\Psi_y(1,t) & \text{ for } 0 < t \leq 1,\\
w(y,0)=0 & \text{ for } 0\leq y \leq 1.
 \end{cases}
 \end{equation}
  First, we regularize the Neumann data so that $w_y$ does not jump at $t=0$.
  For each $j\geq 1$, consider smooth cut-off functions $\zeta_j : [0,1]\to [0,1]$ and $\chi_j : [0,1] \to \R$ such that 
$$
  \zeta_j(t)=
  \begin{cases}
 0 &\text{ if }  0 \leq t \leq \frac{1}{2j},\\
 1 & \text{ if } \frac{1}{j} \leq t \leq 1,
 \end{cases}          
 \quad \text{ and } \quad
   \chi_j(t)= 
 \begin{cases}
 0 & \text{ if }  0 \leq t \leq \frac{1}{2j},\\
\Psi_y(1,t)  &\text{ if }  \frac{1}{j} \leq t \leq 1.
   \end{cases}
$$
Given $\alpha \in (0,1)$, define the {Banach space}:
 $$
 \mathcal{B} := \big \{ u \in C^{0,\alpha}(\overline{Q}) : \|u\|_{L^\infty(Q)}\leq 1/2\big\}.
 $$
We define the mapping
 $$
 \begin{array}{rccl}
 T_j: &  \mathcal{B}  & \to & \mathcal{B}\\
 & \xi & \mapsto & \delta w_y^{\xi, j}, 
 \end{array}
$$
where $w^{\xi,j}$ satisfies the linear uniformly parabolic problem
$$
 \begin{cases}
 w_t - a(\xi) w_{yy} = 0 & \text{ for } 0<y<1, \ 0<t\leq 1,\\
w_y(0,t)= \zeta_j & \text{ for } 0 < t \leq 1,\\
 w_y (1, t) =\chi_j & \text{ for } 0 < t \leq 1,\\
w(y,0)=0 & \text{ for } 0\leq y \leq1,
 \end{cases}
$$
in the classical sense, with
$$
 a(\xi):= \frac{1}{(1+\xi)^2}.
$$
We need to check that $T_j$ is well-defined. First, note that if $\xi \in \mathcal{B}$, then 
 $$\tfrac{1}{4} \leq a(\xi) \leq 4.$$ 
 By the maximum principle (Lemma~\ref{lem:MPmixed}), we have
$$
\|w^{\xi,j}\|_{C(\overline Q)} \leq \sup_{t\in [0,1]} |\zeta_j(t)|+  \sup_{t\in [0,1]} |\chi_j(t)| \leq C_1,
$$
where $C_1>0$ is a universal constant independent of $j$.
Moreover, by classical parabolic theory, there exists a unique solution $w^{\xi, j}\in C^{2,\alpha}(\overline Q)$, with
\begin{equation} \label{eq:badest}
\|w^{\xi, j} \|_{C^{2,\alpha}(\overline Q)} \leq C_0 \big( 1+ \|\zeta_j\|_{C^{1,\alpha}([0,1])} +  \|\chi_j \|_{C^{1,\alpha}([0,1])} \big)\leq C_0(j).
\end{equation}
Differentiating the equation with respect to $y$, we get $\tilde w := w_y^{\xi,j}$ is a weak solution to
$$
\tilde w_t - \big(a(\xi) \tilde w_{y} \big)_y = 0 \quad \text{ in } Q,
$$
and thus, by the maximum principle for the Cauchy-Dirichlet problem \cite[Corollary~6.26]{L},
$$
\|\tilde w\|_{L^\infty(Q)}\leq \|\tilde w\|_{L^\infty(\partial_ p Q)} \leq \sup_{t\in [0,1]} |\zeta_j(t)|+  \sup_{t\in [0,1]} |\chi_j(t)| \leq C_1.
$$
Hence, choosing $\delta>0$ sufficiently small (independent of $j$), we have
$$
\|T_j \xi \|_{L^\infty(Q)} = \delta \| \tilde w\|_{L^\infty(Q)} \leq \delta C_1\leq  1/2. 
$$
Therefore, $T_j \xi \in \mathcal{B}$, and $T_j$ is well-defined.

Moreover, by the estimate \eqref{eq:badest}, we have $T_j$ maps bounded sets in $\mathcal{B}$ into bounded sets in $C^{1,\alpha}(\overline Q)$, which are precompact in $C^{0,\alpha}(\overline Q)$ by Arzel\`a-Ascoli theorem.
Hence, $T_j$ is compact.

It remains to see that $T_j$ is continuous. We need to show that for each $j$ fixed,
$$
\|\xi_k - \xi\|_{C^{0,\alpha}(\overline{Q})} \to 0 \quad \Longrightarrow \quad \|T_j \xi_k - T_j \xi \|_{C^{0,\alpha}(\overline{Q})}  \to 0.
$$
For simplicity, call $w_k=T_j\xi_k$ and $w= T_j \xi$. By construction, $w_k$ satisfies
\begin{equation} \label{eq:wk}
      (w_k)_t - a(\xi_k) (w_k)_{yy} = 0 \quad \text{ in } Q.
\end{equation}
By the $C^{2,\alpha}$ estimate for $w_k$, and Arzel\`a-Ascoli theorem, up to a subsequence, we have
$$
w_k \to \hat w \quad \text{ in } C^2(\overline Q),
$$
for some function $\hat w \in C^{2,\alpha}(\overline Q)$.
Passing to the limit in \eqref{eq:wk}, we get that $\hat w$ satisfies
$$
\hat w_t - a(\xi) \hat w_{yy} = 0 \quad \text{ in } Q.
$$
Moreover, $\hat w_y=w_y$ on $\partial_l Q$ and $\hat w = w$ on $\partial_b Q$. By the comparison principle for the Neumann problem, we see that $\hat w = w$ on $\overline Q$. Therefore, $T_j$ is continuous.

By the Leray-Schauder fixed point theorem \cite[Theorem~11.3]{GT}, for each $j \geq 1$, there exists $w^j\in \mathcal{B}$ such that $T_j w^j = w^j$. Hence, $w^j$ satisfies \eqref{eq:fixedpoint} in the classical sense (in fact, $w^j$ is smooth  by a standard bootstrap argument).
 It remains to pass to the limit as $j\to \infty$.
 
 While the $C^{2,\alpha}$ estimate in \eqref{eq:badest} blows up as $j\to \infty$, the data $\zeta_j$ and $\chi_j$ are uniformly smooth away from the corners $(0,0)$ and $(1,0)$. In fact, we have $\zeta_j = 1$ and $\chi_j = \Psi_y$ for all $t \geq 1/j$. Let $\delta>0$, and let $Q'$ be a compact subset of $\overline Q$ whose parabolic distance to the corners is greater than $\delta$. By classical parabolic estimates, 
$$
\|w^j\|_{C(\overline{Q})} + \delta^{2+\alpha} \|w^j\|_{C^{2,\alpha}({Q'})} \leq C_0 \quad \text{ for $j$ large},
$$
where $C_0>0$ does not depend on $j$.
 By compactness and a diagonal argument, for any compact subset $Q'$ of $\overline Q$, which does not contain the corners, it follows that, up to a subsequence, 
 $$w^j \to w \quad \text{ in } C(\overline{Q})\cap C^2({Q'}).$$
   Passing to the limit in the equation and the initial and boundary conditions, we get that $w$ satisfies \eqref{eq:fixedpoint} in the classical sense.

 \medskip
 \noindent \textbf{Step 2.} \textit{Upper bound at $y=0$.} Let $w$ be as in Step~1, and let $\Psi$ be as in Lemma~\ref{lem:psi}, with $\kappa=4$.
Since $\Psi_{yy}\leq 0$  on $(0,\infty)\times (0,\infty)$, and $a(w) \leq 4$, it follows that
$$
\Psi_t - a(w) \Psi_{yy} \geq \Psi_t - 4 \Psi_{yy} =  0 \quad \text{ in } Q.
$$
Hence, $\bar w:=w-\Psi$ satisfies 
\begin{equation*}
\begin{cases}
\bar w_t - a(w) \bar w_{yy} \leq 0 & \text{ in } Q,\\
\bar w_y = 0 & \text{ on } \partial_l Q,\\
\bar w = 0 & \text{ on } \partial_b Q.
\end{cases}
\end{equation*}
 By the maximum principle (Lemma~\ref{lem:MPmixed})  and Lemma~\ref{lem:psi}, we have $\bar w\leq0$, and thus,
\begin{equation} \label{eq:final}
 w(0,t) \leq \Psi(0,t) = -\frac{4}{\sqrt \pi}\sqrt{t} \quad \text{ for all } t>0.
 \end{equation}
Now, let $\delta>0$ be as in Step~1.
 Then the function 
 $$v:=y+\delta w \quad \text{ on } \overline{Q}$$ 
 is a classical solution to \eqref{eq:optimalsol}. In particular, 
 $v(0,t)=\delta w(0,t)$, for all $t\geq 0$, 
 and thus, \eqref{eq:final} yields the desired estimate. 
  \end{proof}
  
  Finally, we prove the optimal regularity theorem.
  
  \begin{proof}[Proof of Theorem~\ref{thm:optimal}]
  Let $v : \overline{Q}\to \R$ be the function given in the proof of Theorem~\ref{keythm} by $v=y+\delta w$. By construction, we have $\|\delta w_y\|_{L^\infty(Q)}\leq 1/2$. Hence,
  $$
1/2 \leq v_y \leq  3/2 \quad \text{ on } \overline{Q}.
  $$
Since $v$ is monotone increasing with respect to $y$ on $\overline{Q}$, we can consider the change of variables
$$
x= v(y,t) \quad \text{ for } (y,t) \in \overline{Q}.
$$
We define the functions $u(x,t)$ and $\gamma(t)$ as
\begin{align*}
u(v(y,t),t) = y \quad \text{ and } \quad \gamma(t)=v(0,t).
\end{align*}
Recall that $\gamma(0)=v(0,0)=0$. Let $x_1:= \min_{t\in [0,T]} v(1,t)$, where $T\in (0,1]$ is chosen such that $x_1>\gamma(t)$ for all $t\in [0,T]$. Note that $T$ exists since $v(1,0)>\gamma(0)$. Then 
$$
u(x,t) >0 \quad \text{ for } \gamma(t)<x<x_1 \quad \text{ and } \quad u(\gamma(t),t)=0 \quad \text{ for all } t\in [0,T].
$$
Moreover, arguing like in Section~\ref{sec:hodograph}, we see that
$$
u_t - u_{xx} = 0 \quad \text{ for } \gamma(t)<x<x_1.
$$
It remains to solve the problem on the elliptic side. Note that
$$
u_x^+(\gamma(t),t)= \frac{1}{v_y(0,t)} = \frac{1}{1+\delta} \quad \text{ for all } t\in (0,T].
$$
Take $x_0 :=2\min_{t\in [0,T]} \gamma(t)-1$, and define
$$
u(x,t)= \frac{x-\gamma(t)}{1+\delta} \quad \text{ for } x_0<x<\gamma(t).
$$
Then $u$ is a classical solution to \eqref{eq:FBP2}.
Moreover, by Theorems~\ref{thm:interface} and \ref{keythm},
$$
- K \sqrt{t} \leq \gamma(t) \leq - \frac{4\delta}{\sqrt{\pi}} \sqrt{t} \quad \text{ for all } t \in [0,T].
$$
  \end{proof}

  \subsection*{Acknowledgements}
We thank the referees for their careful reading of the manuscript and the helpful comments and suggestions, which have significantly improved the paper.

\subsection*{Funding}
The authors were  supported by the NSF DMS grant 2247096.

\subsection*{Data availability statement} This work has neither used nor created new data.


\end{document}